\documentclass[reqno,12pt]{amsart}

\usepackage{amssymb,amsmath,amsfonts,amssymb}
\usepackage{graphics,graphicx,color}
\def\@abssec#1{\vspace{.05in}\footnotesize \parindent .2in
{\bf #1. }\ignorespaces}

\graphicspath{{/EPSF/}{Figures/}}
\DeclareGraphicsExtensions{.eps}

\newtheorem{theorem}{Theorem}[section]
\newtheorem{lemma}[theorem]{Lemma}
\newtheorem{proposition}[theorem]{Proposition}
\newtheorem{corollary}[theorem]{Corollary}
\newtheorem{definition}[theorem]{Definition}
\newtheorem{remark}[theorem]{Remark}

\def \Rm {\mathbb R}
\def \Nm {\mathbb N}

\def \Sm {\mathbb S}
\newcommand{\eps}{\varepsilon}

\newcommand{\dsum}{\displaystyle\sum}
\newcommand{\dint}{\displaystyle\int}

\newcommand{\pdr}[2]{\dfrac{\partial{#1}}{\partial{#2}}}

\newcommand{\dr}[2]{\dfrac{d{#1}}{d{#2}}}

\newcommand{\mA}{\mathcal A}
\newcommand{\mB}{\mathcal B}
\newcommand{\mC}{\mathcal C}
\newcommand{\mD}{\mathcal D}
\newcommand{\mF}{\mathcal F}
\newcommand{\mG}{\mathcal G}
\newcommand{\mH}{\mathcal H}
\newcommand{\mI}{\mathcal I}

\newcommand{\mL}{\mathcal L}
\newcommand{\mM}{\mathcal M}

\newcommand{\mO}{\mathcal O}
\newcommand{\mP}{\mathcal P}
\newcommand{\mQ}{\mathcal Q}
\newcommand{\mR}{\mathcal R}
\newcommand{\mS}{\mathcal S}
\newcommand{\mT}{\mathcal T}
\newcommand{\mU}{\mathcal U}
\newcommand{\mX}{\mathcal X}
\newcommand{\mY}{\mathcal Y}

\newcommand{\mJ}{{\mathfrak J}}

\renewcommand{\mp}{{\mathfrak p}}

\newcommand{\cout}[1]{}

\newcommand{\dX}{\partial X}

\title{Hybrid inverse problems and redundant systems of partial differential equations}

\author{Guillaume Bal}

\address{Department of Applied Physics and 
        Applied Mathematics, Columbia University, 
        New York NY, 10027; gb2030@columbia.edu}

\begin{document}
 
\maketitle

\begin{center}
\em Dedicated to Gunther Uhlmann's 60th birthday.
\end{center}

\begin{abstract}
Hybrid inverse problems are mathematical descriptions of coupled-physics (also called multi-waves) imaging modalities that aim to combine high resolution with high contrast. The solution of a high-resolution inverse problem, a first step that is not considered in this paper, provides internal information combining unknown parameters and solutions of differential equations. In several settings, the internal information and the differential equations may be described as a redundant system of nonlinear partial differential equations.

We propose a framework to analyze the uniqueness and stability properties of such systems. We consider the case when the linearization of the redundant system is elliptic and with boundary conditions satisfying the Lopatinskii conditions. General theories of elliptic systems then allow us to construct a parametrix for such systems and derive optimal stability estimates.

The injectivity of the nonlinear problem or its linearization is not guaranteed by the ellipticity condition. We revisit unique continuation principles, such as the Holmgren theorem and the uniqueness theorem of Calder\'on, in the context of redundant elliptic systems of equations.

The theory is applied to the case of power density measurements, which are internal functionals of the form $\gamma|\nabla u|^2$ where $\gamma$ is an unknown parameter and $u$ is the solution to the elliptic equation $\nabla\cdot\gamma\nabla u=0$ on a bounded domain with appropriate boundary conditions.
\end{abstract}


\section{Introduction}
\label{sec:intro}

A recent class of (mostly medical) imaging modalities, called hybrid, coupled-physics, or multi-wave modalities offers the possibility to reconstruct high-contrast parameters of interest with high resolution. High contrast is important to discriminate between, say, healthy and non-healthy tissues. Resolution is important to detect anomalies at an early stage. 

Such hybrid modalities typically involve two steps. In the first step, not considered in this paper, a high resolution modality takes as an input measurements performed at the boundary of a domain of interest and provides as an output internal functionals of the parameters of interest and of specific solutions of underlying partial differential equations describing the probing (medical imaging) mechanism. This paper is concerned with the second step, involving the quantitative reconstruction of the parameters from knowledge of said internal functionals. For recent books and reviews on hybrid inverse problems, we refer the reader to, e.g., \cite{A-Sp-08,AS-IP-12,B-IO-12,KK-EJAM-08,S-SP-2011,WW-W-07}.

Most practically used hybrid inverse problems involve internal functionals that are polynomials in the parameters of interest and the specific solutions mentioned above. Combined with the equations describing the latter solutions, we observe that all available information represents a coupled, often redundant, system of nonlinear partial differential equations.

In some instances, local algebraic manipulations allow us to solve such a system explicitly. In the framework of functionals of solutions to second-order equations (and not of their derivatives), we refer the reader to, e.g., \cite{BR-IP-11,BRUZ-IP-11,BU-CPAM-13,BU-IP-10,BU-AML-12}. Such theories find applications in the quantitative step of the imaging modalities Photo-acoustic tomography, Thermo-acoustic tomography, Transient Elastography, and Magnetic Resonance Elastography; see also \cite{CAB-IP-07,CLB-SPIE-09,MY-IP-04,MZM-IP-10,PS-IP-07,SU-IO-12}. In the framework of functionals of the gradients of solutions, which find applications in Ultrasound Modulated tomography and in Current Density Imaging, we refer the reader to, e.g., \cite{ABCTF-SIAP-08,B-APDE-13,BBMT-13,BGM-IP-13,BGM-IPI-13,BM-LINUMOT-13,BS-PRL-10,CFGK-SJIS-09,GS-SIAP-09,KK-AET-11,MB-aniso-13,MB-IP-12,MB-IPI-12}.

In many cases, explicit algebraic inversions may not be known or may not be applicable because not enough information is available. This paper proposes a framework to address several such problems when the linearization of the coupled system is {\em elliptic}. Hybrid inverse problems need not be elliptic; see the example of the $0$-Laplacian in \cite{B-APDE-13,CFGK-SJIS-09} (also recalled below in section \ref{sec:zerolap}) or the Photo-acoustic problem as treated in, e.g., \cite{BR-IP-11,BU-IP-10}. However, when the number of internal functionals increases, the resulting hybrid system becomes more redundant and hence more likely to be elliptic. We consider such a setting in section \ref{sec:ellipticity}. We recall that elliptic systems augmented with boundary conditions that satisfy the Lopatinskii conditions admit left-parametrices. This follows from the theory of Agmon-Douglis-Nirenberg \cite{ADN-CPAM-59,ADN-CPAM-64} and the extensions to redundant systems by Solonnikov \cite{S-JSM-73}. The existence of parametrices allows us to solve the linear problem up to possibly a finite dimensional space. Along with the construction of a parametrix, elliptic regularity theory provides optimal stability results for the linearization of the hybrid inverse problem. 

The analysis of elliptic hybrid inverse problems was first addressed in \cite{KS-IP-12} by means of systems of pseudo-differential operators that were shown to be elliptic in the sense of Douglis and Nirenberg. The differential systems considered in this paper simplify the analysis of boundary conditions and hence of injectivity for the linearized and nonlinear hybrid inverse problems as we now describe.


The possible existence of a finite dimensional kernel for the linearized hybrid inverse problem prevents us from determining whether the available internal functionals uniquely determine the coefficients of interest. Moreover, the dimension of the finite dimensional kernel is not stable with respect to small perturbations, which prevents us from analyzing the uniqueness and stability properties of the nonlinear hybrid problem. A powerful methodology to obtain uniqueness results in the framework of elliptic systems of equations is the notion of {\em unique continuation}. In section \ref{sec:ucp}, we revisit two classical notions of unique continuation. One is based on the Holmgren theorem, which we generalize to the setting of redundant systems considered in this paper. The second one is based on the use of Carleman estimates as they are formulated in Calder\'on's uniqueness theorem. See \cite{C-AJM-58,C-PSFD-62,H-SP-63,H-I-SP-83,H-III-SP-94,L-AM-57,N-AMS-73,Z-Birk-83} for references on these unique continuation results. Several extensions of these results, following the presentation in \cite{N-AMS-73}, are given in the setting of redundant systems in section \ref{sec:ucp} with proofs postponed to the appendix.

Once a reasonable uniqueness result has been obtained for the linearization of the nonlinear hybrid inverse problem, several statements about uniqueness and iterative reconstruction procedures can be formulated for the nonlinear hybrid inverse problem. A constructive fixed point iteration method and a non-constructive local uniqueness result for the nonlinear problem are presented in section \ref{sec:nonlinear}.

As an application of the conditions of ellipticity including boundary conditions and the conditions for unique continuation, we consider the case of power density internal functionals $H_j(x)=\gamma(x)|\nabla u_j|^2(x)$, where $u_j$ is the solution to $\nabla\cdot\gamma\nabla u_j=0$ on an open domain $X\subset\Rm^n$ with boundary conditions $u_j=f_j$ on $\partial X$ for $1\leq j\leq J$. For such a problem, we characterize the conditions under which the redundant problem is elliptic (for $J=2$ in dimension $n=2$ and $J=3$ in higher dimension) and analyze cases in which a unique continuation principle (UCP) applies.

\section{Inverse Problems with local internal functionals.}
\label{sec:ellipticity}

\subsection{Systems of nonlinear partial differential equations}

Let $\gamma$ be a set of constitutive parameters in (linear or nonlinear, scalar or systems of) partial differential equations of the form
\begin{equation}
\label{eq:pde} 
 \mL(\gamma,u_j) =0  \quad \mbox{ in } X,\qquad \tilde \mB u_j = f_j\quad\mbox{ on } \partial X,
\end{equation}
where $\mL$ is a polynomial in the derivatives of the solution $u_j$ and those of $\gamma$ on the open domain $X\subset\Rm^n$, with $u_j$ augmented with boundary conditions on $\partial X$ for $1\leq j\leq J$. 

Let us now assume knowledge of the functionals
\begin{equation}
\label{eq:fct}
\mM(\gamma,u_j) = H_j \quad\mbox{ in } X, \qquad 1\leq j\leq J.
\end{equation}
where $\mM$ is a polynomial in the derivatives of the solution $u_j$ and those of $\gamma$.

Several hybrid inverse problems may be recast in this general framework. More generally, we could have knowledge of functionals of the form $\mM(\gamma,u_i,u_j) = H_{ij}$, or functionals $\mM$ depending on more than two solutions $u_j$. We restrict ourselves to \eqref{eq:fct} to simplify notation.

The above problem may thus be recast as a system of nonlinear partial differential equations for $(\gamma,\{u_j\})$:
\begin{equation}
\label{eq:syst}
\begin{array}{rcl}
\mL(\gamma,u_j) &=&0  \quad \mbox{ in } X,\qquad \tilde \mB u_j = f_j\quad\mbox{ on } \partial X, \qquad 1\leq j\leq J \\
\mM(\gamma,u_j) &=& H_j \quad\mbox{ in } X, \qquad 1\leq j\leq J.
\end{array}
\end{equation}

The first relevant question for the inverse problem is whether the above system admits a {\em unique solution}. Since the solutions $u_j$ are uniquely determined by knowledge of $\gamma$, we are primarily interested in finding a unique solution to the parameters $\gamma$. The strategy followed in \cite{KS-IP-12} consists of writing a system of pseudo-differential equations for $\gamma$. Considering the higher-dimensional coupled system of equations for $(\gamma,\{u_j\})$ allows us to simplify the analysis of the uniqueness question for \eqref{eq:syst}.

The second question pertains to the stability properties of the reconstruction. Provided that the solution to the inverse problem is unique, we wish to understand how perturbations in the information $\{H_j\}$ propagates to the reconstruction of $\gamma$.

The uniqueness and stability properties of the system depend on the number of acquired internal functionals $J$ and on the way the medium was probed via the boundary conditions $\{f_j\}$. Understanding how the uniqueness and stability properties are affected by changes in $J$ and the boundary conditions $\{f_j\}$ is the third question we wish to (very partially) answer.

\subsection{Linearization}

Some problems of the form \eqref{eq:syst} can directly be solved as non-linear systems. For instance, when $\mL$ is a linear second-order equation in $u_j$ and $\mM(\gamma,u_j)=u_j$ the solution itself, the full non-linear problem is analyzed in \cite{BU-CPAM-13,BU-AML-12}.

For many problems in which direct reconstruction procedures may not readily be available, it is fruitful to analyze the linearization of \eqref{eq:syst}. Neglecting boundary conditions at first, this yields
\begin{equation}
\label{eq:linsyst}
\begin{array}{rcl}
\partial_\gamma \mL(\gamma,u_j) \delta\gamma + \partial_u \mL(\gamma,u_j) \delta u_j&=&0  \quad \mbox{ in } X, \qquad 1\leq j\leq J\\
\partial_\gamma \mM(\gamma,u_j) \delta\gamma + \partial_u \mM(\gamma,u_j) \delta u_j &=& \delta H_j \quad\mbox{ in } X, \qquad 1\leq j\leq J.
\end{array}
\end{equation}
Note that the above differential operators may all be of different orders. With $v = (\delta\gamma,\{\delta u_j\})$, we may recast the above system as
\begin{equation}
\label{eq:linear}
\mA v = \mS,
\end{equation}
for an implicitly defined source $\mS$. Let us assume that each $u_j$ is a scalar solution and that each $H_j$ is also a scalar information. Then $\mA$ is a system of differential operators of size $2J\times(J+M)$, where $M$ is the number of scalar functions describing $\gamma$. Note that for $J<M$, the above system is under-determined. We consider here the case $J\geq m$. In several practical problems, $J=M$ gives rise to a determined system $\mA$ that is either not invertible, or invertible with non-optimal stability properties. It is therefore also fruitful to consider the setting with $J>M$.

\subsection{Ellipticity}

In applications, $\mL$ is often a linear, elliptic, operator in the variables $\{u_j\}$. Adding the constraints \eqref{eq:fct}, however, may render the coupled system \eqref{eq:syst} or its linear version \eqref{eq:linsyst} non-elliptic. One fruitful strategy to solve \eqref{eq:linear} on the whole domain $X$ (with appropriate boundary conditions) is therefore to ``ellipticize" $\mA$, i.e., to find a number of constraints $J$ sufficiently large so that $\mA$ is elliptic, provided that such a $J$ exists.

What we mean by elliptic is defined as follows. For each $x\in X$, $\mA(x,D)_{ij}$ is a polynomial in $D=(\partial_{x_1},\ldots,\partial_{x_n})$ for $1\leq i\leq 2J$ and $1\leq j\leq J+M$. We define the {\em principal part} $\mA_0$ of $\mA$ in the sense of Douglis and Nirenberg \cite{DN-CPAM-55}. For each row $1\leq i\leq 2J$ of the system, we associate an integer $s_i$ and for each column $1\leq j\leq J+M$ of the system an integer $t_j$. We normalize these integers by assuming that ${\rm max}(s_i)=0$.

We assume that $\mA_{ij}(x,D)$ is a polynomial in $D$ of degree not greater than $s_i+t_j$.  Then $\mA_{0,ij}(x, D)$ is the part of the polynomial in $\mA_{ij}(x,D)$  of degree exactly equal to $s_i+t_j$. 

When all differential operators in \eqref{eq:linsyst} are of the same order $t$, then we may choose $s_i=0$ and $t_j=t$, in which case $\mA_0(x,D)$ is composed of entries that are homogeneous polynomials of degree $t$ in $D$. Many practical problems arise in forms in which the differential operators in \eqref{eq:linsyst} have different orders.

We say that $\mA$ is {\em elliptic} when the matrix $\mA_0(x,\xi)=\{\mA_{0,ij}(x,\xi)\}$, the symbol of the operator $\mA_0$, is {\em full-rank} (i.e., of rank $J+M$ here) for all $\xi\in\Sm^{n-1}$ the unit sphere and all $x\in\bar X$.

Being full-rank is ``more likely" when $J$ is large, i.e., when $\mA$ is over-determined. It is then useful to acquire redundant information $H_j$ until \eqref{eq:linear} above is elliptic because elliptic systems enjoy more favorable (and in fact optimal) stability estimates than non-elliptic systems.

\subsection{Lopatinskii boundary conditions}
\label{sec:Lop}

Let us assume that we have been able to prove that $\mA$ was a redundant elliptic system of equations. Then the system can be solved, up to possibly a finite dimensional subspace, when it is augmented by boundary conditions that satisfy the Lopatinskii criterion. This is defined as follows; see \cite{S-JSM-73}.

We consider the system
\begin{equation}
\label{eq:linbc}
\mA v = \mS\quad\mbox{ in } X,\qquad \mB v = \phi\quad\mbox{ on } \partial X,
\end{equation}
where $\mB(x,D)$ is a $Q\times (J+M)$ matrix with entries $\mB_{ij}(x,D)$ that are polynomial in $D$ for each $x\in\partial X$.  We denote by $b_{ij}$ the order of $\mB_{ij}$ and by $\sigma_i=\max_j (b_{ij}-t_j)$. Then $\mB_0(x,D)$ is the principal part of $\mB$ and consists of entries $\mB_{0,ij}(x,D)$ defined as the polynomials of $\mB_{ij}(x,D)$ of degree exactly equal to $\sigma_i+t_j$. 

The Lopatinskii conditions are defined as follows. For each $x\in\partial X$, we denote by $\nu(x)$ the outward unit normal to $X$ at $x\in\partial X$. We then think of $z$ as the parameterization of the half line $x-z\nu(x)$ for $z\geq0$. Let $\zeta\in\Sm^{n-1}$ with $\zeta\cdot\nu(x)=0$ and consider the system of ordinary differential equations
\begin{equation}
\label{eq:lopat}
\begin{array}{rcl}
 \mA_0 (x,i\zeta + \nu(x) \dr{}z) u(z) &=& 0 \quad \mbox{ in } z>0 \\[2mm]
 \mB_0 (x,i\zeta + \nu(x) \dr{}z) u(z) &=& 0 \quad \mbox{ at } z=0.
\end{array}
\end{equation}
We assume that for each $x\in\partial X$, the only solution to the above system such that $u(z)\to 0$ as $z\to\infty$ is $u\equiv0$. This is the Lopatinskii condition for $(\mA,\mB)$. We then also say that $\mB$ covers $\mA$.

The above conditions need to be verified for the specific problems being considered. In some situations, the boundary conditions provided by \eqref{eq:pde} (or their linearization) generate a cover of $\mA$. In other situations, they need to be augmented with additional boundary conditions for $\delta u_j$ as well as for $\delta\gamma$ as we shall see.

When $\mA$ is elliptic and $\mB$ covers $\mA$, we say that $(\mA,\mB)$ is an elliptic system.

\subsection{Parametrices and stability estimates.}
\label{sec:param}

Following work in \cite{ADN-CPAM-59,ADN-CPAM-64,DN-CPAM-55} on determined systems, the case of overdetermined elliptic systems was treated in \cite{S-JSM-73}. The salient feature of these works is that the operator $A=(\mA,\mB)$ admits a left-parametrix (a left-regularizer) in the following sense. Let $(\mS,\phi)$ in \eqref{eq:linbc} be in the space
\begin{displaymath}
 \mR(p,l) = W_p^{l-s_1}(X)\times\ldots\times  W_p^{l-s_{2J}}(X) \times W_p^{l-\sigma_1-\frac1p}(\partial X) \times \ldots \times W_p^{l-\sigma_Q-\frac1p}(\partial X),
\end{displaymath}
for some $l\geq0$ and $p>1$ and let us assume that $(\mA,\mB)$ is a bounded operator from 
\begin{displaymath}
 v\in \mU(p,l) = W_p^{l+t_1}(X)\times\ldots\times W_p^{l+t_{J+M}}(X)
\end{displaymath}
to $(\mA,\mB)v=(\mS,\phi)\in\mR(p,l)$. Such is the case when the coefficients of $\mA$ and $\mB$ are sufficiently regular. More precisely, with $l$ sufficiently large so that $p(l-s_i)>n$ for all $1\leq i\leq 2J$ (to simplify; see \cite{S-JSM-73} for slight generalizations), we assume that $\mA_{ij}$ is a sum of homogeneous operators of degree $s_i+t_j-\kappa$ for $0\leq \kappa\leq s_i+t_j$ and that the coefficients of these operators are of class $W^{l-s_i}_p(X)$. Moreover, assuming $l$ large enough so that $p(l-\sigma_q)>n$ as well for $1\leq q\leq Q$, we assume that $\mB_{qj}$ is a sum of homogeneous operators of degree $\sigma_q+t_j-\kappa$ for $0\leq \kappa\leq \sigma_q+t_j$ and that the coefficients of these operators are of class $W^{l-\sigma_q-\frac 1p}(\partial X)$.
 Here $W_p^s(X)$ is the standard Sobolev space of functions with $s$ derivatives that are $p$-integrable in $X$ with standard extensions for $s$ not an integer \cite{adams}.

The main result in \cite{S-JSM-73} is the existence of a bounded operator $R$ from $\mR(p,l)$ to $\mU(p,l)$ such that 
\begin{equation}
\label{eq:leftparam}
RA=I-T,
\end{equation}
where $I$ is the indentity operator and $T$ is compact in $\mU(p,l)$. 
When $1$ is not in the spectrum of $T$ so that $I-T$ is invertible, then $A$ is invertible with bounded inverse $(I-T)^{-1}R$.
However, $1$ could very well be in the spectrum of $T$, in which case ${\rm dim Ker}A$ is finite but positive.

Moreover, we have the following stability estimate
\begin{equation}
\label{eq:stabell}
\dsum_{j=1}^{J+M} \|v_j\|_{W_p^{l+t_j}(X)} \leq C \Big( \dsum_{i=1}^{2J} \|\mS_i\|_{W_p^{l-s_i}(X)} + \dsum_{i=1}^Q \| \phi_i\|_{W_p^{l-\sigma_i-\frac 1p}(\partial X)} \Big) + C_2 \dsum_{t_j>0} \|v_j\|_{L^p(X)},
\end{equation}
for some constants $C>0$ and $C_2>0$. 

The presence of $C_2>0$ indicates the possibility that $A$ may not be invertible. The presence of finite dimensional kernels is a serious difficulty in the analysis of the nonlinear problem \eqref{eq:syst} because such a dimension is not stable with respect to perturbations. What we can ensure is that for $A_1$ sufficiently small, then ${\rm dim Ker}(A+A_1)\leq{\rm dim Ker}A$; see, e.g., \cite{H-III-SP-94}.

Whether we can choose $C_2=0$ above, i.e., whether $A$ is invertible, depends on lower-order terms that are not captured by the principal part $(\mA_0,\mB_0)$. Their analysis can prove quite complicated in practical settings and we do not follow that route here. Instead, our aim is to modify $A$ so that a unique continuation principle may be applied. In section \ref{sec:ucp}, we augment the properly modified system $(\mA,\mB)v=(\mS,\phi)$ with additional boundary conditions, which in some cases allow us to obtain injectivity results. 

Note that the parametrix $R$ is clearly not unique. It is theoretically constructive, as can be seen by following the proof in \cite{ADN-CPAM-64,S-JSM-73}. However, its practical, for instance numerical, implementation is not straightforward. The modified, higher-order, systems proposed later in the section offer a more direct numerical inversion procedure.

\subsection{Example of power-density measurements}
\label{sec:pdm}
%
To illustrate the theoretical result of this paper, we consider the example of the reconstruction of a scalar coefficient from knowledge of the so-called power density measurements. Consider the scalar elliptic equation
\begin{equation}
\label{eq:ellj}
\mL(\gamma, u_j) := \nabla\cdot \gamma\nabla u_j =0 \quad \mbox{ in } \quad X,\qquad u_j = f_j \quad \mbox{ on } \partial X,\qquad 1\leq j\leq J.
\end{equation}
Here, $X$ is an open domain in $\Rm^n$ for $n\geq2$ with smooth boundary $\partial X$.
The objective is to reconstruct the scalar coefficient $\gamma$, uniformly bounded above and below by positive constants, from knowledge of the power densities
\begin{equation}
\label{eq:pdj}
H_j(x) = \mM(\gamma,u_j) := \gamma(x) |\nabla u_j|^2 ,\qquad x\in X,\qquad 1\leq j\leq J,
\end{equation}
where $u_j$ is the solution to \eqref{eq:ellj}.

This problem and some variations have received significant theoretical and numerical analyses in recent years; see, e.g., \cite{ABCTF-SIAP-08,B-APDE-13,BBMT-13,BS-PRL-10,CFGK-SJIS-09,GS-SIAP-09,KK-AET-11,KS-IP-12,MB-IPI-12}; generalizations for anisotropic coefficients $\gamma$ can be found in \cite{BGM-IP-13,BGM-IPI-13,MB-IP-12,MB-aniso-13}. Explicit reconstruction procedures exist when the number of internal functionals $J$ is sufficiently large; see \cite{BBMT-13,BGM-IP-13,CFGK-SJIS-09,MB-IPI-12,MB-aniso-13}. The case $J=1$, which does not correspond to an elliptic system, was analyzed in \cite{B-APDE-13}. The main features of this analysis are recalled in section \ref{sec:zerolap}. For intermediate values of $J$, the above hybrid inverse problem may not have an explicit reconstruction but may still be modeled by a redundant elliptic system. Such a problem was also analyzed in \cite{KK-AET-11,KS-IP-12}. The conditions of ellipticity of the system $(\mA,\mB)$ are described in detail in section \ref{sec:linell}. A modified system is presented in \ref{sec:elimell}, whereas optimal stability estimates of the form \eqref{eq:stabell} are presented in section \ref{sec:stabestim} for the power density measurement problem.

\subsubsection{The $0-$Laplacian when $J=1$}
\label{sec:zerolap}

When $J=1$, the $2\times2$ system of nonlinear partial differential equations is formally determined with two unknown coefficients $(\gamma,u_1)$. The elimination of $\gamma$ from such a system is in fact straightforward and we obtain the equation for $u:=u_1$ (with $H:=H_1$) given by
\begin{equation}
\label{eq:0Lap}
 \nabla\cdot \dfrac{H(x)}{|\nabla u|^2}\nabla u =0 \quad \mbox{ in } \quad X,\qquad u = f \quad \mbox{ on } \partial X.
\end{equation} 
The above equation may be transformed as 
\begin{equation}
  \label{eq:Cauchy2}
  (I-2\widehat{\nabla u}\otimes\widehat{\nabla u}) : \nabla^2 u + \nabla \ln H\cdot\nabla u =0\,\mbox{ in } X,\qquad u=f \,\,\mbox{ and } \,\, \pdr{u}{\nu} = j \,\, \mbox{ on } \partial X.
\end{equation}
Here $\widehat{\nabla u}=\frac{\nabla u}{|\nabla u|}$ and we introduced Cauchy data on $\partial X$ anticipating the fact that \eqref{eq:Cauchy2} is a quasilinear strictly hyperbolic equation, at least provided that $\widehat{\nabla u}$ is defined. Indeed, we observe that the operator $(I-2\widehat{\nabla u}\otimes\widehat{\nabla u}) : \nabla^2$ is hyperbolic with respect to the (unknown) direction $\widehat{\nabla u}$.

The above problem is analyzed in \cite{B-APDE-13}. The two salient features of that analysis are: (i) unique reconstructions of $u$, and hence $\gamma$, are guaranteed only on part of the domain $X$; see \cite{B-APDE-13}; and (ii) the stability estimates are {\em sub-elliptic}: first-order derivatives of $u$ are controlled by the gradient of $H$ rather than second-order derivatives as would be the case if $(I-2\widehat{\nabla u}\otimes\widehat{\nabla u}) : \nabla^2$ was replaced by an elliptic operator. For $\gamma$, this translates into an inequality of the following form. Let $H$ and $\tilde H$ be two measurements corresponding to the pairs $(u,\gamma)$ and $(\tilde u,\tilde\gamma)$, respectively. Assume that the Cauchy data of $u$ and $\tilde u$ agree on $\partial X$. Then we find that on an appropriate (see \cite{B-APDE-13}) subdomain $\mO\subset X$, we have the following stability estimate:
\begin{equation}
\label{eq:subelliptic}
 \|\gamma-\tilde\gamma\|_{L^2(\mO)} \leq C \| \nabla H-\nabla \tilde H\|_{L^2(\mO)}.
\end{equation}
As we shall see below, this estimate is sub-optimal (with a loss of one derivative) when compared to elliptic estimates of the form \eqref{eq:stabell}. It is however optimal for (principally normal) operators of principal type \cite{H-SP-63}.

\subsubsection{Linearization and ellipticity}
\label{sec:linell}

The linearization of the above problem with $J=1$ is a hyperbolic equation. We now wish to show that redundancy in the data ($J\geq2$) allows us to render the system elliptic under some conditions. We consider two ways to obtain elliptic systems of equations.

We first linearize the coupled system \eqref{eq:ellj}-\eqref{eq:pdj} about solutions $(\gamma,u_j)$ and obtain
\begin{equation}
\label{eq:pdlinsyst}
\begin{array}{rcll}
\nabla\cdot \delta\gamma\nabla u_j + \nabla\cdot \gamma\nabla \delta u_j &=& 0 &\quad \mbox{ in } \, X\\
\delta\gamma|\nabla u_j|^2 + 2 \gamma \nabla u_j\cdot\nabla\delta u_j &=& \delta H_j &\quad \mbox{ in } \, X\\
\delta u_j &=& 0 & \quad \mbox{ on } \partial X.
\end{array}
\end{equation}
We define 
\begin{equation}
\label{eq:Fj}
F_j=\nabla u_j
\end{equation}
and {\em assume} that $|F_j|\geq c_0>0$ is bounded from below by a positive constant uniformly. Such an assumption is valid for and appropriate open set of boundary conditions $f_j$; see, e.g., \cite{B-IO-12,BBMT-13,BU-CPAM-13,BU-IP-10} for details of constructions based on complex geometric optics solutions or unique continuation principles, which we do not reproduce here.

Let $\mA_J$ be the operator applied to $\delta v=(\delta \gamma,\{\delta u_j\})$ in the above system. Its principal part $\mP_J$ has for (principal) symbol $\mp_j$ a $2J\times(J+1)$ matrix given by
\begin{equation}
\label{eq:pdpJ}
\mp_J(x,\xi) =  \left(\begin{matrix}
 |F_1|^2  & 2 \gamma F_1\cdot i\xi & \ldots & 0 \\
F_1\cdot i\xi  & -\gamma|\xi|^2 & \ldots & 0 \\
\vdots & \vdots & \ddots & \vdots \\
 |F_J|^2 & 0 & \ldots & 2 \gamma F_J\cdot i\xi\\
F_J\cdot i\xi  & 0 & \ldots & -\gamma|\xi|^2
\end{matrix}\right).
\end{equation}
When $J=1$, the determinant of $\mp_J$ is given by $\gamma|F_1|^2(2(\hat{F_1}\cdot\xi)^2-|\xi|^2)$, which is hyperbolic with respect to $F_1$ as seen in section \ref{sec:zerolap} above. The above system is in Douglis-Nirenberg form for $s_{2k}=1$, $s_{2k+1}=0$, $t_1=0$, $t_j=1$ for $j\geq2$.

When $J\geq1$, we observe that the sub-determinants with the largest number of powers of $|\xi|^2$ are of the form
\begin{displaymath}
   \gamma^{J}|\xi|^{2(J-1)} |F_j|^2 q_j(x,\xi),\qquad q_j(x,\xi) := 2(\hat{F_j}(x)\cdot\xi)^2-|\xi|^2.
\end{displaymath}
We thus obtain that $\mp_J(x,\xi)$ is injective if 
the quadratic forms $q_j(x,\xi)=0$ for all $1\leq j\leq J$ imply that $\xi=0$. 
\begin{definition}
\label{def:ellcondpj}
Define the quadratic forms and operators
\begin{equation}
\label{eq:qjPj}
q_j(x,\xi) = 2 \big(\hat F_j\cdot \xi)^2 - |\xi|^2 ,\qquad P_j(x,D) = \Delta - 2 \hat F_j\otimes\hat F_j : \nabla\otimes \nabla.
\end{equation}
Here $\hat F_j(x)$ are vector fields of unit vectors defined on $\bar X$. We say that the family $\{q_j\}$ or $\{P_j\}$ is elliptic at $x$ if
\begin{equation}
\label{eq:condellip}
 q_j(x,\xi) =0 \mbox{ for all }1\leq j\leq J \quad\mbox{ implies } \quad \xi=0.
\end{equation}
We say that such families are elliptic in $X$ if they are elliptic in at all points $x\in X$. 
\end{definition}
We can then prove the 
\begin{lemma}
\label{lem:ellipAJ}
We assume that $F_j:=\nabla u_j$ is such that $|F_j|$ is bounded from below by a positive constant on $\bar X$ for all $1\leq j\leq J$.

The operator $\mA_J$ defined in \eqref{eq:pdlinsyst} with principal symbol given in \eqref{eq:pdpJ} is elliptic in $\bar X$ if and only if the above family of quadratic forms $\{q_j\}$ is elliptic in $\bar X$.
\end{lemma}
\begin{proof}
  We have already seen the sufficiency of the condition. Let us prove its necessity and assume that $\mp_J(x,\xi)$ is maximal rank for $x\in X$ and $\xi\not=0$. This means that one determinant of $(J+1)\times(J+1)$ sub-matrices of $\mp_J$ is non-vanishing. Each column of $\mp_J$ beyond the first one has two non-vanishing entries. For all $1\leq j\leq J$ except for one entry $j_0$, then either $-\gamma(x)|\xi|^2$ or $2\gamma F_j\cdot i\xi$ appears as a multiplicative factor in the determinant of the sub-matrix. Since $\gamma(x)|\xi|^2$ never vanishes, we may discard the determinant involving $2\gamma F_j\cdot i\xi$. We thus obtain that if one determinant of a sub-matrix does not vanish, then that determinant may be chosen as $\gamma^{J}|\xi|^{2(J-1)} |F_{j_0}|^2 q_{j_0}(x,\xi)$. Since by assumption $\gamma^{J}|\xi|^{2(J-1)} |F_{j_0}|^2$ is bounded away from $0$, we observe that the injectivity of $\mp_J$ implies that (at least) one of the quadratic forms $q_{j}(x,\xi)$ does not vanish. Thus, $\mA_J$ being elliptic implies that $\{q_j\}$ is elliptic in the sense of definition \ref{def:ellcondpj}.
\end{proof}

The ellipticity of $\mA_J$ is thus a consequence of the fact that the null cones of quadratic forms intersect only at $0$. 
We have the following properties:
\begin{proposition}
\label{prop:ell}
 (i) Let $x\in\Rm^2$ and assume that $\hat F_1(x)$ and $\hat F_2(x)$ are neither parallel nor orthogonal. Then the corresponding $(q_1,q_2)$ in \eqref{eq:qjPj} form an elliptic family at $x$. 
 
 (ii) In dimension $n\geq2$, let $\hat F_1(x)$ and $\hat F_2(x)$ be two different directions and define $\hat F_3(x)=\alpha\hat F_1(x)+\beta\hat F_2(x)$ with $\alpha\beta\not=0$ such that $|\hat F_3(x)|=1$. Then the corresponding $(q_1, q_2, q_3)$ in \eqref{eq:qjPj} form an elliptic family at $x$. 
 
 (iii) In dimension $n\geq3$, a family $(q_1,q_2)$ is never elliptic at a given point $x$ independent of the choice of $\hat F_1,\hat F_2$.
\end{proposition}
\begin{proof}
  (i) In dimension $n=2$, it is clear that the null cones (where $q_j$ vanishes, two lines of vectors $\xi$ in $\Rm^2$) coincide if and only if $\hat F_1$ and $\hat F_2$ are either parallel or orthogonal.
  
(iii) The result is obvious if $\hat F_1=\pm \hat F_2$. Assume otherwise and let $F_3$ be a unit vector orthogonal to $\hat F_1$ and $\hat F_2$. Assume $\hat F_1\cdot\hat F_2\geq0$ for otherwise change the sign of $\hat F_2$ (which does not modify the quadratic forms $q_j$). $\xi$ belongs to the intersection of the null cones $\{q_j(x,\xi)=0\}$ for $j=1,2$ if $2(\hat F_1\cdot\xi)^2=2(\hat F_2\cdot\xi)^2=|\xi|^2$. Define $\xi=\hat F_1+\hat F_2+\lambda F_3$. The first constraint is satisfied when $2(1+\hat F_1\cdot\hat F_2)^2 = |\hat F^1+\hat F^2|^2+\lambda^2$, i.e., when  $\lambda=\pm \big(2\hat F_1\cdot\hat F_2+ 2(\hat F_1\cdot\hat F_2)^2)^{\frac12}$. But then it is clear that this $\xi\not=0$ also belongs to the second cone so that $(q_1,q_2)$ is not elliptic.

(ii) Let us assume that $\hat F_3=\alpha \hat F_1+\beta \hat F_2$ with $\alpha\beta\not=0$. Then $\xi$ belongs to the three null cones if 
\begin{displaymath}
 \Big(\dfrac{\alpha \hat F_1+\beta \hat F_2}{|\alpha \hat F_1+\beta \hat F_2|}\cdot\xi\Big)^2 = \dfrac12 |\xi|^2 = (\hat F_1\cdot \xi) ^2 = (\hat F_2\cdot\xi)^2.
\end{displaymath}
Expanding the first constraint, we get
\begin{displaymath}
 \alpha^2(\hat F_1\cdot\xi)^2 + \beta^2(\hat F_2\cdot\xi)^2 + 2\alpha\beta \hat F_1\cdot\xi\hat F_2\cdot\xi = \frac12|\xi|^2(\alpha^2+\beta^2+2\alpha\beta \hat F_1\cdot\hat F_2).
\end{displaymath}
The last two constraints imply that $2\hat F_1\cdot\xi\hat F_2\cdot\xi=\eps|\xi|^2$ with $\eps=\pm1$. Combined with the latter, they also imply that
\begin{displaymath}
\alpha\beta \eps = \alpha\beta \hat F_1\cdot\hat F_2.
\end{displaymath}
Since $\alpha\beta\not=0$, this can only occur if $\hat F_1=\pm\hat F_2$, which is a contradiction.
\end{proof}

From the practical point of view, this result says that if $F_1=\nabla u_1$ and $F_2=\nabla u_2$ are not parallel, then the internal functionals $H_j$ for $u_1$, $u_2$, and $u_1+u_2$ (with $F_3=\nabla (u_1+u_2)=\nabla u_1+\nabla u_2$) generate three quadratic forms $q_j$, $1\leq j\leq 3$ that form an elliptic family. Such internal functionals are obtained by choosing three boundary conditions of the form $f_1$, $f_2$, and $f_3=f_1+f_2$. This result holds for all $n\geq2$.

\subsubsection{Sufficient conditions for ellipticity} We have thus obtained the following result. In dimension $n=2$, $J\geq2$ is necessary for $\mA_J$ to be elliptic. Moreover, $J=2$ is sufficient when $n=2$ if $\nabla u_1$ and $\nabla u_2$ are nowhere parallel or orthogonal. In dimension $n\geq3$, $J\geq3$ is necessary for $\mA_J$ to be elliptic. Moreover, $J=3$ is sufficient for $\mA_J$ to be elliptic in all dimensions $n\geq2$ by choosing as boundary conditions, e.g.,  $(f_1,f_2,f_1+f_2)$ provided that $\nabla u_1$ and $\nabla u_2$ are nowhere parallel.

\subsubsection{Boundary conditions and Lopatinskii condition.} In order to obtain an optimal theory of stability estimates, the system needs to be augmented with boundary conditions that satisfy the Lopatinskii condition. Dirichlet conditions on $\delta u_j$ and no condition on $\delta\gamma$ satisfy such conditions. Indeed, we need to show that $v(z)=(\delta\gamma(z),\ldots,\delta u_J(z))\equiv 0$ is the only solution to
 \begin{equation}\label{eq:pdLP}
 \begin{array}{l}
\delta u_j(0)=0,\qquad    (iF_j\cdot\zeta+F_j\cdot N \partial_z) \delta\gamma + \gamma (i\zeta+\partial_z)^2 \delta u_j =0 ,\,\\|F_j|^2 \delta\gamma +2\gamma(iF_j\cdot\zeta+ F_j\cdot N \partial_z) \delta u_j=0,\,\, z>0
\end{array}
\end{equation}
with $v(z)$ vanishing as $z\to\infty$ for $N=\nu(x)$ at $x\in\partial X$ and $z$ coordinate along $-N$. Eliminating $\delta\gamma$ as earlier, we deduce that 
\begin{displaymath}
 \Big( |F_j|^2 (i\zeta+\partial_z)^2 - 2 (iF_j\cdot\zeta+F_j\cdot N\partial_z)^2 \Big) \delta u_j=0
\end{displaymath}
The leading term of the above second order equation with constant coefficients is $|F_j|^2q_j(x,N)\partial_z^2$.  If $q_j(x,N)\not=0$ for some $j=j_0$, which is the condition for joint ellipticity described in definition \ref{def:ellcondpj}, then the same proof showing that Dirichlet conditions cover the Laplace operator show that $\delta u_{j_0}=0$. We then deduce that $\delta\gamma=0$ from the second line in \eqref{eq:pdLP} and by ellipticity on the first line in \eqref{eq:pdLP} that all $\delta u_j=0$ and hence $v\equiv0$.

Let us define the spaces $\mX=W^l_{p}(X)\times W^{l+1}_p(X;\Rm^J)$ and $\mY^i=W^l_p(X;\Rm^{2J})$ with $l$ large enough so that the latter spaces are all algebras; i.e., $pl>n$. We also define $\mY^\partial = W^{l+1-\frac 1p}_p(X;\Rm^{J})$ for the traces of $\delta u_j$ on $\partial X$, which vanish by construction. Then we observe that $(\mA,\mB)$ in \eqref{eq:pdlinsyst} maps $\mU(p,l)=\mX$ to $\mR(p,l)=\mY^i\times\mY^\partial$. Moreover, the coefficients $\gamma$ and $u_j$ appearing in the definition of $\mA$ belong to $W^l_{p}(X)$ and $W^{l+1}_p(X;\Rm^J)$, respectively, by assumption for $\gamma$ and by elliptic regularity for $u_j$ solution of the elliptic problem \eqref{eq:ellj}.

\subsubsection{Elimination and ellipticity}
\label{sec:elimell}

The above system involves $J+1$ unknowns. Strategies to lower the dimension of the system include: (i) eliminating $\gamma$ as we did in obtaining \eqref{eq:0Lap}; or (ii) eliminating all $u_j$. The second strategy follows from the observation
\begin{displaymath}
   \delta u_j = L_\gamma^{-1} (\nabla\cdot \delta\gamma \nabla u_j),\qquad L_\gamma = -\nabla\cdot\gamma \nabla,
\end{displaymath}
with $L_\gamma^{-1}$ defined by solving $L_\gamma$ with vanishing Dirichlet conditions.

The result is a redundant system of the form $P_j\delta\gamma=\delta H_j$ where $P_j$ is a pseudo-differential operator with principal symbol given by $|F_j|^2q_j(x,\xi)$. The redundant system for $\delta\gamma$ is therefore elliptic under the same conditions as those for \eqref{eq:pdpJ} above. The main difficulty is that $P_j$ is no longer local (no longer a (system of) partial differential equation). It is not clear how one may approach the question of uniqueness for such a system. The unique continuation principles presented in section \ref{sec:ucp} do not apply directly. See \cite{KS-IP-12} for an analysis of such a method.

The first strategy based on the elimination of $\delta\gamma$ preserves the differential structure of the original system since $\delta\gamma$ appears undifferentiated in the second equation in \eqref{eq:pdlinsyst}. We find that
\begin{equation}
\label{eq:deltauj}
\nabla\cdot\gamma\big(-\nabla \delta u_j + 2 \hat F_j \hat F_j\cdot\nabla \delta u_j\big) = \nabla\cdot \dfrac{\delta H_j}{|F_j|^2} F_j.
\end{equation}
The symbol for such an equation is then given by $\gamma q_j(x,\xi)$ so that the above equation is not elliptic, as we already know. However, the elimination of $\delta\gamma$ also provides the constraint
\begin{equation}
\label{eq:constraintjk}
2\gamma\dfrac{1}{|F_j|^2} F_j\cdot\nabla \delta u_j - \dfrac{\delta H_j}{|F_j|^2}  =  2\gamma\dfrac{1}{|F_k|^2} F_k\cdot\nabla \delta u_k - \dfrac{\delta H_k}{|F_k|^2} \quad 1\leq j<k\leq J.
\end{equation}
It turns out that the combination of \eqref{eq:deltauj} with \eqref{eq:constraintjk} makes the redundant system for the $\{u_j\}$ elliptic provided $q_j(x,\xi)=0$ for all $j$ implies that $\xi=0$ as above. To see this, assume that 
\begin{displaymath}
  q_j(x,\xi) \delta u_j =0,\qquad |F_k|^2F_j\cdot\xi \delta u_j = |F_j|^2 F_k\cdot\xi \delta u_k.
\end{displaymath}
Let $\xi\not=0$. Then, not all $q_j(x,\xi)$ vanish. Assume that $q_1(x,\xi)\not=0$. Then $\delta u_1=0$ so that $F_j\cdot\xi \delta u_j=0$ for all $j\geq2$. However, $F_j\cdot\xi$ and $q_j(x,\xi)$ cannot vanish at the same time so that $\delta u_j=0$. This shows the injectivity of the symbol of the redundant system, which is easily found to be in Douglis-Nirenberg form. 

Moreover, unlike the redundant system of strategy (ii), the above system \eqref{eq:deltauj}-\eqref{eq:constraintjk} is differential. It may therefore be augmented with boundary conditions that satisfy the Lopatisnkii conditions. We leave the details to the reader to verify, as we did in \eqref{eq:pdLP}, that such conditions are satisfied for Dirichlet conditions $\delta u_j=0$ on $\partial X$ under the same conditions guaranteeing ellipticity in $X$. The above system for the $(\{\delta u_j\})$ is therefore elliptic when the family of quadratic forms $\{q_j\}$ is elliptic.

\subsubsection{Stability estimates}
\label{sec:stabestim}

Both systems of differential equations for $(\delta\gamma,\{\delta u_j\})$ and for $(\{\delta u_j\})$ after elimination of $\delta\gamma$ are elliptic with Dirichlet conditions for $\delta u_j$ under the conditions stated, e.g., in Proposition \ref{prop:ell}. We may therefore apply the general theory of elliptic redundant systems described above and obtain a parametrix for both systems. Moreover, in both cases, we obtain the following stability estimates
\begin{equation}
\label{eq:pdstab}
  \|\delta\gamma-\delta\tilde\gamma\|_{W_p^l(X)} + \dsum_{j=1}^J \|\delta u_j-\delta\tilde u_j\|_{W_p^{l+1}(X)} \leq C \dsum_{j=1}^J \|\delta H_j-\delta \tilde H_j\|_{W_p^l(X)} +C_2 \dsum_{j=1}^J \|\delta u_j-\delta\tilde u_j\|_{L_p(X)}.
\end{equation}
In other words, $\delta\gamma$ and $\nabla \delta u_j$ are reconstructed from $\delta H_j$ with {\em no loss} of derivative, unlike the construction in \eqref{eq:subelliptic}. However, we are not guaranteed that any of the linear systems is indeed invertible, and hence the presence of the last term on the right-hand-size in \eqref{eq:pdstab}. Proving the injectivity of the above systems is much more delicate than proving their ellipticity. The following section proposes some generic strategies to do so. 

\subsubsection{Generalization to similar models}

The results presented above generalize to the setting where the internal functionals are of the form
\begin{equation}
\label{eq:alphaH}
H_j(x) = \gamma^{\alpha} |\nabla u_j|^2,
\end{equation}
where $\alpha\geq0$. The case $\alpha=1$ was treated above. The case $\alpha=2$ corresponds to the setting of CDII and MREIT \cite{NTT-Rev-11,SW-SR-11}. As shown, e.g., in \cite{KS-IP-12,MB-IPI-12}, the coupled problem is elliptic with $J=1$ for $\alpha>2$, is degenerate elliptic for $\alpha=2$, and is hyperbolic for $\alpha<2$. The cases $\alpha\geq2$ become elliptic for $J$ chosen sufficiently large. We do not consider these extensions further here.

\subsection{Sufficient conditions for ellipticity}
\label{sec:ellipticitycond}

To summarize the above derivation, we have modeled the nonlinear hybrid inverse problem as a coupled system of nonlinear partial differential equations \eqref{eq:syst}. Its linearization is then given by \eqref{eq:linsyst}. These systems have $J+M$ unknowns for $2J\geq J+M$ equations. When $J$ is large compared to $M$, we expect the principal part $\mA_0$ of $\mA$ in the Douglis-Nirenberg sense to be elliptic since a redundant matrix is ``more likely" to be full-rank than a less elongated matrix. 


The matrix $\mA_0$ is full rank if we can prove that, for well chosen boundary conditions $f_j$, the internal functionals $H_j$ are sufficiently independent. Such a result is problem-dependent. In the setting of power-density measurements, we have obtained that the boundary conditions $f_j$ were well-chosen if the quadratic form $q_j(x,\xi)$ were jointly elliptic at each point $x\in \bar X$. A sufficient condition to do so is to choose two boundary conditions $f_1$ and $f_2$ such that $F_1=\nabla u_1$ and $F_2=\nabla u_2$ are nowhere co-linear (for the quadratic forms corresponding to the three boundary conditions $f_1$, $f_2$, and $f_1+f_2$ are then jointly elliptic).

Strategies to ensure that, e.g., $\nabla u_1$ and $\nabla u_2$ are nowhere co-linear have been presented in, e.g., \cite{B-APDE-13,BBMT-13,BU-IP-10,BU-CPAM-13}, see also the review \cite{B-IO-12}. Such strategies are based on the use of Complex Geometrical Optics when the latter are available, or on the use on local construction and unique continuation properties of the operators in \eqref{eq:pde} as described in \cite{BU-CPAM-13}. In both settings, it is proved that qualitative properties of solutions to elliptic equations, such as for instance the independence of gradients of solutions, hold for an open set of boundary conditions $f_j$. More explicit constructions of such boundary conditions are also proposed in \cite{BC-JDE-13}.

%
\section{Injectivity results for the linearized problem}
\label{sec:ucp}
%

In this section, we consider two methods to obtain injectivity of an elliptic operator $\mA$ augmented with appropriate boundary conditions $\mB$. 
Both are based on replacing the redundant system $\mA$ by its normal, determined, form $\mA^t\mA$, and augmenting it with appropriate Dirichlet boundary conditions that ensure its injectivity under certain assumptions. The first method invokes a Holmgren unique continuation principle while the second method is based on unique continuation principles that are consequences of Carleman estimates.

Here, the operator $\mA^t$ is defined such that $(\mA^t)_{ki}=(\mA_{ik})^t=:\mA^t_{ik}$, where $(\mA_{ik})^t$ is the formal adjoint to $\mA_{ik}$ for the usual inner product $(\cdot,\cdot)$ on $L^2(X)$.

One difficulty with operators $\mA$ that are elliptic in the Douglis-Nirenberg (DN) sense is that the normal operator $\mA^t\mA$ need not be elliptic, even in the DN sense. Consider for instance the $2\times2$ system in one independent variable defined by $\mA_{11}=a$, $\mA_{12}=\mA_{21}=\partial_x$ and $\mA_{22}=\partial^2_x$, which is DN elliptic with $s_1=t_1=0$ and $s_2=t_2=1$ when $a\not=1$. Defining $\mC=\mA^T\mA$, the principal term in $\mC$ is $\mC_{11}=\partial_x^2$, $\mC_{12}=\mC_{21}=\partial_x^3$, $\mC_{22}=\partial_x^4$, which is independent of $a$ and not elliptic.

One way to ensure that $\mA^t\mA$ is elliptic when $\mA$ is elliptic is to assume that $\tau=t_j$ independent of $j$ and $s_i=0$ independent of $i$.  We then verify that the leading term in $(\mA^t\mA)_{kl}$ is a differential operator of degree equal to $2\tau$ and that  $\mA^t\mA$ is a strongly elliptic system of size $(J+M)\times(J+M)$; see, e.g., \cite[Proposition 4.1.16]{T-AV-95}. The principal part of ${\rm Det}(\mA^t\mA)(x,\xi)$ is equal to ${\rm Det}(\mA_0^t\mA_0)(x,\xi)$ and is a polynomial of degree $2(J+M)\tau$ that is uniformly bounded from below for all $x\in \bar X$ and $\xi\in\Sm^{n-1}$ by assumption of ellipticity.

The results in \cite{C-JMAA-91} show that an elliptic system in DN form can always be transformed into an elliptic overdetermined system of first-order equations. The procedure increases the order of the system and differentiates any row involving a term with $s_i+t_j=0$. This forces us to impose boundary conditions on the parameters $\delta\gamma$ and the solutions $\delta u_j$ that may not be necessary in the definition of $(\mA,\mB)$ above. 

For the rest of the section, we {\em assume} that $\mA$ has been recast into a form where $\tau=t_j$ is independent of $j$ and $s_i=0$ is independent of $i$; for instance with $\tau=1$ or $\tau=2$ as described in \cite{C-JMAA-91}. We still keep the notation $2J$ and $J+M$ for the size of the system $\mA$ so that with this notation, $\mA^t\mA$ is a system of size $(J+M)\times(J+M)$.



Let us augment $\mA^t\mA$ with the Dirichlet boundary conditions
\begin{equation}
\label{eq:Dirichlet} 
\big(\pdr{}\nu\big)^q v_j =\phi_{qj}\quad\mbox{ on } \partial X, \qquad 0\leq q\leq \tau-1,\quad 1\leq j\leq J+M.
\end{equation}
We recast the above constraints as $\mD v=\phi$ on $\partial X$. It is proved in \cite[p.43-44]{ADN-CPAM-64} that such boundary conditions cover $\mA^t\mA$, i.e., that the Lopatinskii conditions are satisfied. We thus consider the problem
\begin{equation}
\label{eq:normal}
\mA^t\mA v = \mA^t\mS\quad\mbox{ in } X,\qquad \mD v = \phi\quad\mbox{ on } \partial X.
\end{equation}
Since $N:=(\mA^t\mA,\mD)$ is elliptic, the above system admits a left parametrix $G$ such that $GN=I-T$ with $T$ compact  in $\mU(p,l)$. Moreover, we have the stability estimate
\begin{equation}
\label{eq:stabnormal}
\dsum_{j=1}^{J+M} \|v_j\|_{W_p^{l+\tau}(X)} \leq C \Big( \dsum_{j=1}^{J+M} \|(\mA^t\mS)_j\|_{W_p^{l-\tau}(X)} + \dsum_{j,q} \| \phi_{qj}\|_{W_p^{l-\tau+q-\frac 1p}(\partial X)} \Big) + C_2 \dsum_{j=1}^{J+M} \|v_j\|_{L^p(X)},
\end{equation}
We verify that 
\begin{displaymath}
  \|\sum_{i=1}^{2J} \mA^t_{ji} S_i \|_{W_p^{l-\tau}(X)} \leq C \sum_{i=1}^{2J} \|S_i\|_{W_p^{l}(X)},
\end{displaymath}
so that \eqref{eq:stabnormal} may also be seen as an analog of \eqref{eq:stabell}.

Our objective is to find sufficient conditions under which the above system is injective, and hence invertible, so that $C_2=0$ in the above estimates. We start with the following simple lemma:
\begin{lemma}
\label{lem:Normal}
Let us assume that $v$ is a solution of 
\begin{equation}
\label{eq:normal0}
\mA^t\mA v = 0 \quad\mbox{ in } X,\qquad \mD v = 0\quad\mbox{ on } \partial X.
\end{equation}
Then $v$ is a solution of 
\begin{equation}
\label{eq:syst0}
\mA v = 0 \quad\mbox{ in } X,\qquad \mD v = 0\quad\mbox{ on } \partial X.
\end{equation}
\end{lemma}
Note that the Dirichlet conditions associated to $\mA^t\mA$ are now redundant for \eqref{eq:syst0}. For example, consider $\mA=\Delta+c(x)$ a scalar operator. For some choices of $c(x)$ (for instance $c(x)=\lambda$ an eigenvalue of $-\Delta$ on $X$ with Dirichlet conditions), $\mA$ is not invertible. However, \eqref{eq:syst0} corresponds to $(\Delta+c(x))v=0$ with both $v=0$ and $\partial_\nu v=0$ on $\partial X$. It is then known that the unique solution to the above constraints is $v=0$ and hence $(\mA^t\mA,\mD)$ is injective.
\begin{proof}
 We observe that the differential operator $(\mA^t)_{ki}=(\mA_{ik})^t=:\mA^t_{ik}$ is of the same order $\tau$ as $\mA_{ik}$. Since $\partial_\nu^q v_k=0$ for $0\leq q\leq \tau-1$, we obtain by integrations by parts that
 \begin{displaymath}
 0=  \sum_{ijk}(\mA^t_{ik}\mA_{ij} v_j,v_k) = \sum_i (\sum_j \mA_{ij} v_j, \sum_k \mA_{ik} v_k) =\sum_i \|(\mA v)_i\|^2.
\end{displaymath}
This implies the result.
\end{proof}
%
\subsection{Holmgren unique continuation and generic results}
\label{sec:ucph}
%

The above result shows that the injectivity of \eqref{eq:normal} is a consequence of the injectivity of \eqref{eq:syst0} with redundant boundary conditions. The operator $\mA$ is therefore augmented with more boundary conditions than is necessary to render it elliptic. This is a possible price to pay to guarantee that the solution to the normal system \eqref{eq:normal} is uniquely defined.

In general, however, even with redundant boundary conditions, it is not entirely straightforward to ensure that $v=0$ is the only solution to \eqref{eq:syst0}. In this section, we consider the case where $\mA$ is well-approximated by an operator with analytic coefficients.  We have the following result adapted from \cite{H-I-SP-83}.
\begin{proposition}\label{prop:WFanalytic}
 Let $\mA_A(x,D)$ be a $2J\times (J+M)$ system of differential equations with analytic coefficients in $X$ such  that $\mA_A^t \mA_A(x,D)$ is an elliptic operator for all $x\in X$. Assume that $\mA_A v=0$ in $X$. Then $v$ is analytic in $X$.
\end{proposition}
\begin{proof}
  Since $\mA_Av=0$, we also have that
  \begin{displaymath}
   ^{co}(\mA_A^t\mA_A) (\mA_A^t\mA_A) v = {\rm Diag}\big({\rm det} (\mA_A^t \mA_A),\ldots,{\rm det} (\mA_A^t \mA_A) \big) v =0,
\end{displaymath}
where $^{co}(\mA_A^t\mA_A)$ is the matrix of co-factors of $\mA_A^t\mA_A$. Since $\mA_A^t \mA_A$ is elliptic (a consequence of the fact that the leading term in $\mA_A(x,\xi)$ is of full rank $J+M$ for all $(x,\xi)\in\bar X\times\Sm^{n-1}$), then $P={\rm det} (\mA_A^t \mA_A)$ is an elliptic operator such that ${\rm Char}P=\emptyset$, where the characteristic set of $P$ is $\{(x,\xi)\in \bar X\times \Sm^{n-1},\, P_m(x,\xi)=0\}$ and $P_m$ is the principal part of $P$. We deduce  from \cite[Theorem 6.1]{H-I-SP-83} that $WF_A(v_j)\subset {\rm Char}P \cup WF_A(Pv_j)=\emptyset$ so that each $v_j$ is analytic in $X$.
\end{proof}

From this, we deduce the following unique continuation result:
\begin{theorem}[Holmgren]
\label{thm:holmgren}
Let $\mA_A$ be as in Proposition \ref{prop:WFanalytic} and let us assume that $\mA_A v=0$ in $X$. Then we have:

(i) Assume that $v=0$ in an open set $\Omega\subset\subset X$. Then $v=0$ in $X$.

(ii) Assume that $\mD v=0$ on an open set $\Sigma$ of $\partial X$. Then $v=0$ in $X$.
\end{theorem}
\begin{proof}
  Let us first prove (i). We know from Proposition \ref{prop:WFanalytic} that the functions $v_j$ are analytic. Since they vanish on an open set, they have to vanish everywhere.
  
  Now (ii) is a standard consequence of (i), which we write in detail for systems. Let $x_0\in \Sigma$ and $V$ a sufficiently small open ball around $x_0$ where the coefficients of $\mA_A$ are analytic and where $\mD v=0$ on $\Sigma\cap V\subset\Sigma$. Since $\mA_A$ is injective, we deduce that $(\partial_n)^{\tau_j} v_j=0$ on $\Sigma\cap V$ as well, as for any higher-order derivative in fact. Let us extend $v$ by $0$ on $V$. Then we verify that $\mA_A v=0$ in $V$ (all that needed verification was that the equation was satisfied point-wise on $\Sigma\cap V$). But now (i) implies that $v=0$ on $V\cup X$.
\end{proof}

The above result shows that $A_A=\mA_A^t\mA$ is injective as well. Since $R_AA_A=I-T_A$ for some left-parametrix $R_A$ and compact operator $T_A$, we deduce that $1$ is not in the spectrum of $T_A$ so that $A_A^{-1}=(I-T_A)^{-1}R_A$. As a consequence, any operator $A$ sufficiently close to $A_A$ is also invertible with a bounded inverse. We  summarize this result as:
\begin{corollary}
\label{cor:generic} Let $\mA=\mA_A+\mA_1$ with $\mA_1$ sufficiently small in operator norm from $\prod_i W_p^{l}(X)$ to $\prod_j W_p^{l+\tau}(X)$. Then $\mA$ augmented with $\mD v=0$ is injective and $\mA^t\mA$ augmented with the same boundary conditions is invertible with bounded inverse. In other words, \eqref{eq:stabnormal} holds with $C_2=0$.
\end{corollary}
The invertibility of $\mA^t\mA$ is therefore {\em generic}, i.e., holds for an open and dense set of coefficients in the definition of the elliptic operator $\mA^t\mA$; see \cite{SU-JFA-09}. Note, however, that the size of norm of $\mA_1$ for which $\mA$ is invertible depends on $\mA_A$.

\medskip

Another similar result states that the invertibility of $\mA^t\mA$ is guaranteed when $X$ is a sufficiently small domain.
\begin{theorem}
\label{thm:small} Let us assume that $\mA(0,D)$ is an elliptic operator. Then for $X$ sufficiently small, we have that $\mA$ is injective and that $\mA^t\mA$ is invertible when augmented with Dirichlet conditions $\mD v=0$. As a consequence, \eqref{eq:stabnormal} holds with $C_2=0$.
\end{theorem}
 As in the previous corollary, the size of the domain $X$ depends on the bound for $(\mA(0,D)^t\mA(0,D))^{-1}$ with Dirichlet conditions on $\partial X$, which is independent of $X$ as we shall see. It is therefore possible to estimate the size of the domain $X$ for which the above theory applies; see also \cite{BM-LINUMOT-13} for a similar result with a different method of proof. 
\begin{proof}
 We assume that all coefficients of $\mA$ are sufficiently smooth and that $0\in X$. Let $P=\mA(0,D)$ be the operator of order $m=\tau$ with coefficients frozen at $x_0=0$. Then $P^tP$ with Dirichlet conditions on $\partial X$  is invertible as an application of Theorem \ref{thm:holmgren}. For constant coefficients, a more precise theory applies. Let $u\in H^m(X)$ (we restrict ourselves to Hilbert spaces to simplify notation) be such that $u=\partial^j_\nu u=0$ for $1\leq j\leq m-1$, i.e., $\mD u=0$. We assume $\partial X$ sufficiently smooth that we can extend $u$ by $0$ outside of $X$ and obtain a function in $H^m(\Rm^n)$ with the same norm. Let $Y$ be an open set such that $\bar X\subset Y\subset\Rm^n$. 
 
 The elliptic theory for constant coefficient operators \cite{H-I-SP-83} provides the existence of a fundamental solution $E$ such that
\begin{displaymath}
  E * P^*Pu = u .
\end{displaymath}
This implies that 
\begin{displaymath}
   \|u\|_{H^{m}(X)} = \|u\|_{H^{m}(Y)} = \|E * P^*Pu\|_{H^{m}(Y)} \leq\|E*P^*\|_{{\mathcal L}(L^2(X),H^m(Y))} \|Pu\|_{L^2(X)}.
\end{displaymath}
This proves the existence of a constant $C$ independent of $X$ such that 
\begin{displaymath}
   \|Pu\|_{L^2(X)} \geq C  \|u\|_{H^{m}(X)}.
\end{displaymath}
This holds for any operator with constant coefficients. 

Now let $Q$ be an operator of order at most $m$ and of small norm $\eps$ from $L^2(X)$ to $H^m(X)$. Then
\begin{displaymath}
 \|(P+Q)u\|_{L^2(X)} \geq (C -\eps) \|u\|_{H^m(X)}\geq C_1 \|u\|_{H^m(X)},
\end{displaymath}
which implies that $P+Q$ is injective for $\eps$ sufficiently small.

This is applied as follows. Let $P=\mA(x,D)$ be an elliptic operator with not-necessarily constant coefficients and define $P_0=\mA(0,D)$ the operator with coefficients frozen at $x_0=0$. Then $P-P_0$ is an operator bounded from $H^m(X_\eps)$ to $L^2(X_\eps)$ with bound $C_1\eps$ for $C_1$ independent of $\eps$ and $X_\eps\subset B(x_0,\eps)$. From the above, we deduce that
\begin{displaymath}
  \|P_0 u\|_{L^2(X_\eps)}\geq C \|u\|_{H^m(X)}
\end{displaymath}
for all function $u\in H^m(X)$ such that $\mD u=0$. For $C_1\eps<C$, we find that
\begin{displaymath}
  \|P u\|_{L^2(X_\eps)}\geq C_2 \|u\|_{H^m(X)}
\end{displaymath}
for some $C_2>0$. This proves that $P=\mA$ is injective on sufficiently small domains.
\end{proof}

%
\subsection{Unique continuation principle}
\label{sec:ucpc}
%
When the domain $X$ is large or when the operator $\mA$ is not sufficiently close to an operator $\mA_A$ with analytic coefficients, then the unique continuation results must rely on an other principle than analyticity. 

Unique continuation in the absence of analyticity in the coefficients is not always guaranteed, even for scalar elliptic operators, although the construction of counter-examples is not straightforward. For general references on unique continuation principles (UCP) for mostly scalar, not necessarily elliptic, operators, as well as counter-examples, see \cite{H-III-SP-94,N-CPAM-57,N-AMS-73,Z-Birk-83} and their references.

In this section, we revisit Calder\'on's \cite{C-AJM-58,C-PSFD-62} result of uniqueness from Cauchy boundary data closely following the presentation in \cite{N-AMS-73} and extend it to redundant systems. Our objective is to adapt \cite[Theorem 5]{N-AMS-73} to specific settings of redundant systems. We first present local uniqueness results, whose proofs are postponed to the appendix, and then use classical arguments to extend them to global uniqueness results. 

\subsubsection{Local uniqueness result}
\label{sec:lur}

Let $x$ be a point in $\Rm^{n+1}$ and $N$ be a unit vector equal to $(0,\ldots,0,1)$ in an appropriate system of coordinates. We want to address the local uniqueness of the Cauchy problem. Assume that $L_q$ for $1\leq q\leq Q$ are differential operators of order $m$ and that
\begin{equation}
\label{eq:redundsyst}
 L_q u =0 \mbox{ in } V\cap \{x_{n+1}>0\},\qquad \partial^j_{n+1} u=0 \mbox{ on } V\cap \{x_{n+1}=0\},
\end{equation}
for $1\leq q\leq Q$ and $1\leq j\leq m-1$, 
where $V$ is a neighborhood of $0$. We assume that $N$ is non characteristic (at $x$) for all $L_q$, i.e., $p_q(x,N)\not=0$ with $p_q$ the principal symbol of $L_q$. This and \eqref{eq:redundsyst} imply that all derivatives of $u$ vanish on $\{x_{n+1}=0\}$ and that $u$ can be extended by $0$ on $V\cap \{x_{n+1}<0\}$. The uniqueness problem may therefore be recast as: if $u$ satisfies
\begin{equation}
\label{eq:UCPsyst}
 L_q u =0 \mbox{ in } V \quad \mbox{ for } 1\leq q\leq Q\qquad \mbox{ and } \qquad u=0 \mbox{ in } \{x_{n+1}<0\},
\end{equation}
then $u\equiv0$ in a full neighborhood of $0$. 

When $Q=1$, it is known that $L=L_1$ needs to satisfy several restrictive assumptions in order for the result to hold; see \cite{N-AMS-73}. The main advantage of the redundancy in the above system is that such assumptions need to be valid only locally (in the Fourier variable $\xi$) for each operator, and globally collectively.

Changing notation, we define $t=x_{n+1}$ and still call $x=(x_1,\ldots,x_n)$. Then $p_q=p_q(x,t,\xi,\tau)$ is the principal symbol of $L_q$ for this choice of coordinates. Sufficient conditions for the uniqueness to the above Cauchy problem involve the properties of the roots $\tau$ of the above polynomials $\tau\mapsto p_q(\cdot,\tau)$ as a function of $\xi$. In the setting of redundant measurements, these conditions may hold for different values of $q$ for different values of $\xi$. This justifies the following definition of our assumptions. 

We assume the existence of a finite covering $\{\Omega_\nu\}$  of the unit sphere $\Sm^{n-1}$ (corresponding to $|\xi|=1$) such that the following holds. For each $\nu$, there exists $q=q(\nu)$ and $\eps>0$ such that for each $(x,t)$ close to $0$ and each $\xi\in\Omega_\nu$, we have:
\begin{equation*}
\label{eq:condCalderon}
\begin{array}{rl}
(i) & p_q(x,t,\xi,\tau) \mbox{  has at most simple real roots $\tau$ and at most double complex roots} \\
(ii) & \mbox{distinct roots $\tau_1$ and $\tau_2$ satisfy } |\tau_1-\tau_2|\geq\eps>0\\
(iii) & \mbox{non-real roots $\tau$ satisfy } |\Im \tau|\geq\eps.
\end{array}
\end{equation*}
Then we have the following result:
\begin{theorem}
\label{thm:calderonsystem}[Calder\'on's result for redundant systems.]
Assume that $N$ is non characteristic for the operators $L_q$ at the origin and that for a finite covering $\{\Omega_\nu\}$ of the unit sphere, (i)-(ii)-(iii) above are satisfied.   Then \eqref{eq:UCPsyst} implies that $u=0$ in a full neighborhood of $0$. 
\end{theorem}
The proof, which  closely follows that of Theorem 5 in \cite{N-AMS-73}, is presented in the appendix.

In the above theorem, $u$ is scalar. The above proof extends to vector-valued functions $u$ when the system for $u$ is diagonally dominant. For a determined system given by a matrix $L_{ij}$ of operators of order $m\geq1$, this means that the operators $L_{ii}$ are of order $m$ and satisfy the hypotheses of the above theorem and the operators $L_{ij}$ for $i\not=j$ are of order at most $m-1$.  This generalizes to the setting of redundant systems $L_{ij}u_j=S_i$ where for each $i$, only one operator $L_{ij}$ is of order $m$ and for each $j$, the operators $L_{kj}$ of order $m$ collectively satisfy the hypotheses of the above theorem; see Theorem \ref{thm:calderonsystemred:app} in the appendix, which we do not reproduce here.

The above theorem may not apply directly in applications. However, it serves as a component to obtain more general results. We consider one such result that finds applications in the framework of power density measurements. 

We consider a setting where the leading term in the system may not be diagonal but rather upper-triangular. Unique continuation properties may still be valid provided that (a sufficiently large number of) the diagonal operators are elliptic. However, the corresponding complex roots may no longer be double. Instead, we need the stronger condition for some $\eps>0$:
\begin{equation*}
\label{eq:condCalderon2}
\begin{array}{rl}
(iv) & p_q(x,t,\xi,\tau) \mbox{ has at most simple roots $\tau$} \mbox{ that satisfy } |\Im \tau|\geq\eps.
\end{array}
\end{equation*}

Then we have the following result.
\begin{theorem}
\label{thm:2by2system}
Consider the redundant system of equations
\begin{equation}
\label{eq:2by2syst}
\left(\begin{matrix} L_1 & L_0 \\ L_3 & L_2 \end{matrix} \right) 
\left(\begin{matrix} u_1 \\ u_2\end{matrix} \right) = 0\mbox{ in } V\cap \{x_{n+1}>0\},\qquad \partial^j_{n+1} u_k=0 \mbox{ on } V\cap \{x_{n+1}=0\},
\end{equation}
for $1\leq j\leq m-1$ and $1\leq k\leq 2$, with the following assumptions. The operators $L_1$ and $L_0$ are (vector-valued) $Q_1\times1$ operators of order $m$,  where $L_1$ satisfies the hypotheses of Theorem \ref{thm:calderonsystem} with $Q=Q_1$. The operators $L_2$ and $L_3$ are $Q_2\times1$ operators of order $m$ and at most $m-1$, respectively. Moreover, $L_2$ satisfies the ellipticity hypothesis (i)-(ii)-(iv) with $Q=Q_2$.
Then $(u_1,u_2)=0$ in a full neighborhood of $0$.

The same result extends to systems of the form
\begin{equation}
\label{eq:RbyRsyst}
\left(L_{ij} \right)_{1\leq i,j\leq R} u = 0\mbox{ in } V\cap \{x_{n+1}>0\},\qquad \partial^j_{n+1} u_k=0 \mbox{ on } V\cap \{x_{n+1}=0\},
\end{equation}
for $1\leq k\leq R$, where $L_{ij}$ is (a vector-valued operator) of order $m-1$ for $i>j$, $L_{ii}$  is (a vector-valued operator) of order $m$ that satisfies the hypotheses of Theorem \ref{thm:calderonsystem} with an appropriate value of $Q$ when all $L_{ik}$ are of order $m-1$ for $k\not=i$ and the ellipticity hypothesis (i)-(ii)-(iv)  with an appropriate value of $Q$ when at least one operator $L_{ik}$ for $k<i$ is of order $m$.
\end{theorem}
The proof of this theorem is also given in the appendix. Note that {\em (i)-(ii)-(iv)}, as opposed to {\em (i)-(ii)-(iii)}, may be seen as an ellipticity condition since the symbol $p_{q(\nu)}$ does not vanish on $\Omega_\nu$. It is known that Carleman estimates are sharper in such a setting; see, e.g., \cite{LR-JPCS-11,N-AMS-73}. 

\begin{remark}
\label{rem:LC}\rm Let $\mQ$ be an invertible $Q_1\times Q_1$ matrix and define the linear operators $\tilde L_1=\mQ L_1$ and $\tilde L_0=\mQ L_0$. Then the conclusions of Theorem \ref{thm:2by2system} are the same as those obtained when $L_1$ and $L_0$ are replaced by $\tilde L_1$ and $\tilde L_0$. In other words, the results of Theorem \ref{thm:calderonsystem} hold if $L_q$ is replaced by an equivalent linear combination. The condition that $N$ be non characteristic for $L_q$ may then be replaced by the condition that it be non characteristic for the  linear combinations $\tilde L_q=\mQ_{q,q'} L_{q'}$.
\end{remark}

\begin{remark}\label{rem:extensions} \rm
 If $L_0$ is an operator of order $m-1$ in the latter theorem so that the above system is diagonally dominant, then it is sufficient to assume that  $L_1$ and $L_2$ satisfy the less constraining hypotheses of Theorem \ref{thm:calderonsystem}; see Theorem \ref{thm:calderonsystemred:app} and the discussion below Theorem \ref{thm:calderonsystem}.
\end{remark}

%
\subsubsection{Global uniqueness result}
%

The above local results extend to global unique continuation results such as \cite[Sections 8.5-8.6]{H-I-SP-83}:
\begin{theorem}
\label{thm:globalUCP}
 Let $\mA$ be a (redundant) system of operators of order $m$, let $R_{\mA}\subset\{(x,N)\in \bar X\times \Sm^{n-1}\}$ be the set of points where the hypotheses of either Theorem \ref{thm:calderonsystem} or Theorem \ref{thm:2by2system}, (and possibly their extensions in Remark \ref{rem:extensions}) are satisfied and let us define $\Sigma_{\mA}= \bar X\times\Sm^{n-1} \backslash R_{\mA}$. Let
 \begin{displaymath}
 \mA u =0 \quad \mbox{ in } X,\qquad \partial^j_\nu u =0 \quad \mbox{ on } \partial X, \,\, 0\leq j\leq m-1.
\end{displaymath}
Let $\bar N({\rm supp}(u))$ be the subset of  $(\bar x,\xi)\in X\times\Sm^{n-1}$ composed of the closure of the normal set of the support of $\{u_j\}$; see \cite{H-I-SP-83}. Heuristically, that set may be seen as the points $(x,\xi)$ where $x$ belongs to the boundary of ${\rm supp (u)}$ and $\xi$ is a normal (inward or outward) unit vector to that boundary.

Then $\bar N({\rm supp}(u))\subset \Sigma_{\mA}$. When the latter set is empty, then $u\equiv0$ in $X$.
\end{theorem}
The proof follows from the local results noting that the boundary to the support of $u$ cannot occur at a point and for a direction where a local UCP principle holds;  see \cite[Sections 8.5-8.6]{H-I-SP-83} for the definitions and derivations. 


\begin{remark}\rm
 Let $\mA$ be a $p\times q$ system with $p\geq q$ and let us assume that $\mA u=0$ in the above theorem is replaced by 
\begin{equation}
\label{eq:globalucpconst}
|(\mA u)_k|(x) \leq C \dsum_{|\alpha|\leq m-1} \dsum_{l=1}^q |D^\alpha u_l|(x),
\end{equation}
for a constant $C$ independent of $x\in \bar X$ and $1\leq k\leq p$. Let, as above,  $\Sigma_\mA$ be the set of points where the hypotheses of Theorem \ref{thm:calderonsystem} do not hold (with obvious modifications for  Theorems \ref{thm:2by2system}).  Then the conclusions of the above theorem, namely that $\bar N({\rm supp}(u))\subset \Sigma_{\mA}$, still hold; see, e.g., \cite{H-III-SP-94} and the proofs in the appendix.
\end{remark}

\subsection{Application to power-density measurements}
\label{sec:pduniq}

We now apply the unique continuation results of sections \ref{sec:ucph} and \ref{sec:ucpc} to the setting of power density measurements described in section \ref{sec:pdm}.

\subsubsection{Second-order linear systems of equations}

Both unique continuation results require that the system $\mA$ be elliptic in the regular sense, i.e., with $\tau=t_j$ independent of $j$ and $s_i=0$ for all $i$. The systems \eqref{eq:pdlinsyst} and \eqref{eq:deltauj}-\eqref{eq:constraintjk} are not in this form. We propose two modifications of the latter systems for which uniqueness results can be proven.

\subsubsection{System after elimination.}
Let us first consider the system \eqref{eq:deltauj}-\eqref{eq:constraintjk}. The latter constraints are first-order. However, they easily become second-order by differentiation. We thus consider the system
\begin{equation}
\label{eq:elim2nd}
\begin{array}{rcl}
  \nabla\cdot\gamma\big(-\nabla \delta u_j + 2 \hat F_j \hat F_j\cdot\nabla \delta u_j\big) &=& \nabla\cdot \dfrac{\delta H_j}{|F_j|^2} F_j.\\[3mm]
\nabla \big( 2\gamma\dfrac{1}{|F_j|^2} F_j\cdot\nabla \delta u_j - 2\gamma\dfrac{1}{|F_k|^2} F_k\cdot\nabla \delta u_k\big)  &=&   \nabla \big( \dfrac{\delta H_j}{|F_j|^2} - \dfrac{\delta H_k}{|F_k|^2}\big) \quad 1\leq j<k\leq J,
\end{array}
\end{equation}
in $X$ with boundary conditions $\delta u_j=\partial_\nu \delta u_j=0$ on $\dX$. The above system is clearly elliptic under the same hypotheses as  \eqref{eq:deltauj}-\eqref{eq:constraintjk}. Moreover, it satisfies the above ellipticity condition with $\tau=2$.

\subsubsection{System in triangular form.}
Let us come back to \eqref{eq:pdlinsyst}. We apply the operator $K_j=2\gamma |F_j|^{-2}F_j\cdot\nabla$ to the first line and $L_\gamma |F_j|^{-2}$ to the second line with $L_\gamma:=\nabla\cdot \gamma\nabla$ to obtain after subtraction
\begin{equation}
\label{eq:pdlinsyst2}
\begin{array}{rcl}
\nabla\cdot \delta\gamma F_j + \nabla\cdot \gamma\nabla \delta u_j &=& 0 \\
L_\gamma \delta\gamma - K_j \nabla\cdot \delta\gamma F_j +[L_\gamma K_j,K_j L_\gamma] \delta u_j &=& L_\gamma \dfrac{\delta H_j}{|F_j|^2}.
\end{array}
\end{equation}
Since $L_\gamma$ is second order and $K_j$ is first-order, the commutator $[L_\gamma K_j,K_j L_\gamma]$ is a second-order differential operator. We thus again obtain a system that is elliptic when \eqref{eq:pdlinsyst} is (we leave this proof to the reader; see also \eqref{eq:triangsyst} below) and that is in the appropriate form with $\tau=2$.

\subsubsection{Holmgren unique continuation result}

The results of section \ref{sec:ucph} apply to the power density measurement systems \eqref{eq:elim2nd} and \eqref{eq:pdlinsyst2}. Let $\mA$ denote the linear operator in one of the aforementioned systems. Then $\mA^t\mA$ augmented with Dirichlet conditions is injective provided that $\gamma$ and $u_j$ are sufficiently close to analytic coefficients or provided that the problem is posed on a sufficiently small domain, and provided of course that the quadratic forms $\{q_j\}$ are collectively elliptic on $\bar X$ in the sense of Definition \ref{def:ellcondpj}.

The main difference between \eqref{eq:elim2nd}  and \eqref{eq:pdlinsyst2} pertains to the boundary conditions. In the case of \eqref{eq:pdlinsyst2}, vanishing Dirichlet and Neumann boundary conditions are imposed on $\delta\gamma$ and $\delta u_j$. This means that the value of $(\gamma,\{u_j\})$ and of its (normal) derivatives is known on $\partial X$. In the case of \eqref{eq:elim2nd}, only the values of $(\{u_j\})$ and of its (normal) derivatives need to be known. 

In fact, the boundary conditions for $\delta\gamma$ at $x\in\partial X$ may be deduced from those for $\delta u_1$, say, provided that $\nu(x)$ is non characteristic for $P_1$. Indeed, we observe that 
\begin{displaymath}
  \delta H_1 = \delta\gamma |\nabla u_1|^2 + 2 \gamma \nabla u_j\cdot\nabla \delta u_1.
\end{displaymath}
This provides boundary conditions for $\delta\gamma$ when $\delta u_1$ and $\partial_\nu  \delta  u_1$ are known on $\partial X$. The derivation of $\partial_\nu \delta\gamma$ is obtained as follows. Using the equation
\begin{displaymath}
   \nabla\cdot \delta\gamma \nabla u_j + \nabla\cdot \gamma\nabla \delta u_j =0,
\end{displaymath}
and replacing $\delta\gamma$ above using the expression of $\delta H_1$, we obtain after straightforward eliminations that
\begin{displaymath}
   (1-2 (\widehat{\nabla u_1} \cdot \nu(x))^2) \partial^2_\nu u_1 =f_1,
\end{displaymath}
where $f_1$ is a known expression involving $\partial_\nu u_1$ and tangential derivatives of $u_1$ of degree up to $2$ at $x$, which are known since $u_1$ is known on $\partial X$. Since $\nu$ is non characteristic for $P_1$, this provides an expression for $\partial^2_\nu u_1$, and hence for $\partial_\nu\delta\gamma$ using the above expression for $\delta H_1$.

So knowledge of Cauchy data for $\delta\gamma$ is in fact a consequence of knowledge of Cauchy data for the solutions $\delta u_j$. Thus, the major requirement in solving $\mA^t\mA v=\mA^t\mS$ is that we impose that $\partial_\nu\delta u_j=g_j$ for a known function $g_j$, which implies that $\partial_\nu u_j$ is known on $\dX$.

\subsubsection{Calder\'on unique continuation result}
The symbol of the operator \eqref{eq:pdlinsyst} is elliptic in the Douglis-Nirenberg sense. Recall that  the results of unique continuation principle (UCP) of the preceding sections apply to operators that are elliptic in a classical sense. We thus consider the  modification of the system \eqref{eq:pdlinsyst2}. We were not able to obtain a satisfactory global UCP using Carleman-type estimates for the system \eqref{eq:elim2nd}.
   
UCP for such  the modified system \eqref{eq:pdlinsyst2} is based on proving a collective UCP for quadratic forms, which we now address in detail.
\begin{lemma}
\label{lem:ucphyper}
Let $P_j(x,D)$ be second-order operators with principal symbols given by $q_j(x,\xi)=2(\hat F_j\cdot\xi)^2-|\xi|^2$, where $\hat F_j(x)\in\Sm^{n-1}$ for $1\leq j\leq J$. 
Let $(x,N)\in\bar X\times\Sm^{n-1}$. For $\xi'\in\Rm^n$ such that $\xi'\cdot N=0$, we define
\begin{displaymath}
  q_{j;N}(x,\xi') = 2(\hat F_j\cdot\xi')^2 - \big(1-2(\hat F_j\cdot N)^2\big) |\xi'|^2.
\end{displaymath}
Let us assume that the quadratic forms are independent in the following UCP sense:
\begin{equation}\label{eq:UCPN}
  q_{j;N}(x,\xi') =0 \quad \mbox{ for all } 1\leq j\leq J\mbox{ s.t. } q_j(x,N)\not=0\qquad \mbox{ implies that } \qquad \xi'=0.
\end{equation}

Then $\{P_j\}$ collectively satisfies a UCP for $(x,N)$ in the sense that we can find a partition $\{\Omega_\nu\}$ such that  (i)-(ii)-(iii) holds.
\end{lemma}

\begin{proof}
The proof in dimension $n=2$ is trivial: $P_j$ satisfies UCP except for four directions on the null cone, i.e., for $\xi$ such that $q_j(x,\xi)=0$. Since \eqref{eq:UCPN} implies that $N$ is not in the intersection of the null cones, then UCP holds for one of the operators $P_j$ and we can choose $\Omega_1=\Sm^{1}$ in the partition.

 Let $(x,N)\in\bar X\times\Sm^{n-1}$ and $n\geq3$. Let $\xi'\in\Sm^{n-1}$ such that $\xi'\cdot N=0$. Then
 \begin{displaymath}
  q_j(x,\xi'+\tau N) = 2(\hat F_j\cdot (\xi'+\tau N))^2- |\xi'+\tau N|^2 = \big(2(\hat F_j\cdot N)^2-1)\tau^2 + 4 \hat F_j\cdot N\hat F_j\cdot \xi' \tau + 2 (\hat F_j\cdot \xi')^2-1.
\end{displaymath}
This quadratic polynomial (because $q_j(x,N)=2(\hat F_j\cdot N)^2-1\not=0$) with real-valued coefficients has a double root when the determinant
\begin{equation}\label{eq:Deltaprime}
 \Delta'_j = -1+2 \big((\hat F_j\cdot N)^2 + (\hat F_j\cdot \xi')^2\big) = q_{j;N}(x,\xi')=0.
\end{equation}
For $\xi'$ away from such points, then the UCP condition holds for $P_j$ since the roots cannot be simple and complex roots (which come in conjugate pairs) cannot be close to each-other. By assumption, $\Delta'_j=0$ for all $j$ is not possible and by continuity, $\max_{j}|\Delta'_j|$ is bounded away from $0$ so that different roots $\tau$ are separated by some $\eps>0$ (as well as the imaginary part of such roots since they come in conjugate pairs). This generates a finite partition $\{\Omega_\nu\}$. On each patch, one $P_j$ satisfies the UCP properties {\em (i)-(ii)-(iii)}.
\end{proof}

In dimension $n=2$, ellipticity of $(P_1,P_2)$ is equivalent to UCP of $(P_1,P_2)$ at each $(x,N)$. In dimension $n\geq3$, collective ellipticity and collective UCP are different notions. As soon as there is one $j$ such that $2(\hat F_j\cdot N)^2\geq1$, the UCP is satisfied at $(x,\xi)$ while ellipticity does not necessarily holds.
However, we have the result:
\begin{lemma}
\label{lem:ElltoUCP}
 In dimension $n=2$, collective UCP and collective ellipticity of operators of the form $P_j$ above is equivalent. In dimension $n=3$, collective ellipticity of $\{P_j\}$ such that $q_j(N)\not=0$ implies collective UCP at $(x,N)$.
\end{lemma}
\begin{proof}
 The case $n=2$ is handled as in the proof of Lemma \ref{lem:ucphyper}. Consider now the case $n=3$. Non-UCP at $(x,N)$ implies from \eqref{eq:Deltaprime} and the fact that $N$ is non characteristic for $P_j$ that
 \begin{displaymath}
   (\hat F_j\cdot N)^2 + (\hat F_j\cdot \xi')^2 = \frac12.
\end{displaymath}
If we decompose $\hat F_j=\hat F_j\cdot N N + \hat F_j\cdot \xi' \xi' + \hat F_j"$ with $\hat F_j"$ orthogonal to ${\rm span}(N,\xi')$, then we find that 
\begin{displaymath}
  |\hat F_j"|^2 = \frac 12.
\end{displaymath} 
In dimension $n=3$, we find that $\hat F_j"=\pm|\hat F_j| N\times \xi'$. This implies that $N\times\xi'$ belongs to the null cone of $q_j$, i.e., $q_j(x,N\times\xi')=0$. This holds for every $j$, which implies that $N\times \xi'=0$, a contradiction.
\end{proof}

\begin{definition}\label{def:UCP}
We say that the family of operators $P_j(x,D)$ or of quadratic forms $q_j(x,\xi)$ collectively satisfy a global UCP on $\bar X$ if  for every $(x,N)\in \bar X\times \Sm^{n-1}$, we can find a family $\tilde q_j(x,\xi)=\mQ_{jk}q_k(x,\xi)$ with $\mQ$ an invertible matrix such that $\{\tilde q_j(x,\xi)\}$ collectively satisfies a UCP at $(x,N)$ in the sense of Lemma \ref{lem:ucphyper}.
\end{definition}

The above definition involves a linear change of quadratic forms for each $(x,N)$ that allows us to obtain a family of modified forms $\tilde q_j$ with similar properties to those of $q_j$ and such that $\tilde q_j(N)\not=0$; see Remark \ref{rem:LC}. This will prove useful in the analysis of reconstruction from  power density functionals. We will also need the following Lemma in dimension $n=3$.
\begin{lemma}
\label{lem:3dUCP}
Let $F_1$ and $F_2$ be three-dimensional vector fields on $\bar X$ such that for each $x\in\bar X$, ${\rm rank}(F_1,F_2)=2$. Let $F_3=F_1+F_2$. We denote by $\{P_j\}$ and $\{q_j\}$ the corresponding operators and quadratic forms defined in \eqref{eq:qjPj} for $1\leq j\leq 3$. Then $\{q_j\}$ satisfies a global UCP on $\bar X$.
\end{lemma}
\begin{proof}
  Let us fix $(x,N)\in \bar X\times \Sm^2$ and choose a basis of $\Rm^3$ such that  $\hat F_1(x)=e_1$ and ${\rm span}(F_1,F_2)=(e_1,e_2)$. The three quadratic forms $q_j$ allow us to obtain by linear combination any $q_c$ corresponding to $F_c=c e_1+se_2$ with $s=\sqrt{1-c^2}$. It is then not difficult to find three values of $c$ such that the corresponding $\tilde q_j$ for $1\leq j\leq 3$ form an elliptic family and such that $\tilde q_j(N)\not=0$. From Lemma \ref{lem:ElltoUCP}, we obtain that $\{\tilde q_j\}$ satisfies a UCP at $(x,N)$. The hypotheses of global UCP in Definition \ref{def:UCP} are met.
\end{proof}

%
%

\subsubsection{Elliptic system in triangular form.} We now present a modification of the linear system for $(\delta\gamma,\{\delta u_j\})$ for which a global UCP result as described in the above definition can be obtained.

We recast \eqref{eq:pdlinsyst2} as
\begin{equation}
\label{eq:triangsyst}
\left(\begin{matrix} P_j & \tilde P_j \\ 0 & \Delta \end{matrix} \right) \left(\begin{matrix} \delta\gamma \\ \delta u_j \end{matrix} \right) = l.o.t.\left(\begin{matrix} \delta\gamma \\ \delta u_j \end{matrix} \right)  + S_j,
\end{equation}
where $l.o.t.$ means a system of differential operators of order at most $1$. Let $P$ and $\tilde P$ denote the $J\times 1$ columns of the second-order, homogeneous, operators $P_j$ and $\tilde P_j$, respectively. The symbol of $P_j$ is proportional to $q_j(x,\xi)$.  The operator $\Delta$ satisfies hypothesis {\em (iv)} of Theorem \ref{thm:2by2system} while the operator $\tilde P$ is of order $2$. If the operators $P$ collectively satisfy the hypotheses of Theorem \ref{thm:calderonsystem}, then the above system \eqref{eq:pdlinsyst2} satisfies a UCP, as does the fourth-order operator $\mA^t\mA$ as constructed in the preceding section:

\begin{theorem}
\label{thm:uniqpowdens}
Let $\mA v= \mS$ in $X$ be the system for $v=(\delta\gamma,\{\delta u_j\})$ described in \eqref{eq:triangsyst} and augmented with boundary conditions $v=\partial_\nu v=0$ on $\partial X$. Assume that the coefficients in $\mA$ are sufficiently smooth (see proof of Theorem \ref{thm:calderonsystem:app}). Let the operators $P_j$ above satisfy a UCP on $\bar X$ as described in Definition \ref{def:UCP} and Lemma \ref{lem:ucphyper}. Then any solution to that system, or to the system $\mA^t\mA v=\mA^t \mS$ with the same boundary conditions, is unique.
\end{theorem}
\begin{proof}
  The proof is a direct application of Theorem \ref{thm:2by2system} since $\Delta$ satisfies hypothesis {\em (iv)}, $\tilde P$ is of order $2$, and $P$ collectively satisfies a UCP at each $(x,N)$ in $\bar X\times \Sm^{n-1}$. 
\end{proof}

In dimensions $n=2$ and $n=3$, we have the following sufficient conditions on the internal functionals to guaranty injectivity of $\mA^t\mA$:
\begin{corollary}\label{cor:UCP}
\label{cor:3d} Assume that $(u_1,u_2)$ are solutions such that $F_1=\nabla u_1$ and $F_2=\nabla u_2$ are such that ${\rm rank}(F_1,F_2)=2$ for each $x\in\bar X$.  

In dimension $n=2$, assume moreover that  $F_1\cdot F_2\not=0$ for each $x\in\bar X$. Then $\{P_1,P_2\}$ satisfies a global UCP property on $\bar X$ and the results of Theorem \ref{thm:uniqpowdens} hold.

In dimension $n=3$, define $F_3=F_1+F_2$. Then $\{P_1,P_2,P_3\}$ satisfies a global UCP property on $\bar X$ and the results of Theorem \ref{thm:uniqpowdens} hold.
\end{corollary}

The proof is immediate using Lemma \ref{lem:3dUCP}, remark \ref{rem:LC} and the preceding Theorem.

We have thus exhibited a system of equations  $\mA^t\mA v=\mA^t \mS$ with boundary conditions $v$ and $\partial_\nu v$ prescribed on $\partial X$, which admits a unique solution that verifies the stability estimate \eqref{eq:pdstab} (generalized as in \eqref{eq:stabnormal} for non-homogeneous boundary conditions) with $C_2=0$.

As in the preceding section devoted to the Holmgren uniqueness result, this result comes at the cost of having to impose boundary conditions of the form $v$ and $\partial_\nu v$ to both $v=\delta u_j$ and $v=\delta\gamma$. As we saw in the preceding section, the boundary conditions for $\delta\gamma$ at $x\in\partial X$ may be deduced from those for $\delta u_j$.

Comparing ellipticity and uniqueness criteria for $\mA$ in \eqref{eq:triangsyst}, we observe that ellipticity is obtained by imposing conditions $\delta u_j=g_j$ on $\dX$ and ensuring that $\{q_j\}$ is an elliptic family of quadratic forms. UCP is obtained by imposing the additional boundary conditions $\partial_\nu\delta u_j=\psi_j$ and by ensuring that $\{q_j\}$ satisfy a global UCP on $\bar X$. In dimension $n=2$, global ellipticity and global UCP for $\{q_j\}$ are equivalent. In dimension $n=3$, both are satisfied for the family of three internal functionals described in Corollary \ref{cor:UCP}.


%
\section{Local uniqueness for nonlinear inverse problem}
\label{sec:nonlinear}
%

Once the injectivity of a linear system can be established, standard theories may be applied to obtain local uniqueness results for the nonlinear inverse problems. Let us recast the nonlinear system of PDE \eqref{eq:syst} as
\begin{equation}
\label{eq:nonlinearnotmod}
\begin{array}{rcl}
 \tilde\mF (v) &=& \tilde \mH \quad \mbox{ in } X \\
 \tilde \mB v &=& \tilde g \quad \mbox{ on } \partial X.
\end{array}
\end{equation}
Here $v=(\gamma,\{u_j\}_{1\leq j\leq J})$ and $\tilde \mB$ is an operator that maps $v$ to the trace of $\{u_j\}$ on $\partial X$. We assume here, as is the case for hybrid inverse problems, that the boundary condition operator $\tilde \mB$ is linear.

The linearization of the above system involves the operator $\tilde\mF'(v_0)$ for a fixed $v_0$. The analysis of the preceding sections did not allow us to show that the differential operator $\tilde \mF'(v_0)$ augmented with the boundary conditions $\tilde \mB$ was invertible. Rather, we obtained that for an appropriate linear differential operator $\mT(v_0)$, then $\mA:=\mT(v_0)\tilde \mF'(v_0)$ augmented with the augmented boundary conditions $\mB$ was invertible. More precisely, we obtained that
\begin{equation}
\label{eq:linearmod}
\begin{array}{rcl}
  \mA w := \mT(v_0) \tilde \mF'(v_0) w &=& \mS  \quad \mbox{ in } X \\
 \mB w &=&  h\quad \mbox{ on } \partial X,
\end{array}
\end{equation}
admitted a unique solution with a stability estimate given by
\begin{equation}
\label{eq:stablinear}
  \|w\|_\mX \leq C \|(\mS,h) \|_\mY.
\end{equation}

The above linearized operator is the linearization of the nonlinear operator
\begin{equation}
\label{eq:nonlinearmF}
\mF(v_0+w):=\mT(v_0)\tilde \mF(v_0+w).
\end{equation}
 We thus modify the original inverse problem and replace it with an inverse problem that is necessarily satisfied by the exact solution $v=v_0+w$ we wish to reconstruct and is given by
\begin{equation}
\label{eq:nonlinearmod2}
\begin{array}{rcl}
 \mF (v_0+w) &=& \mH:=\mT(v_0) \tilde\mH \quad \mbox{ in } X \\
 \mB (v_0+w) &=& g \quad \mbox{ on } \partial X.
\end{array}
\end{equation}
Borrowing the notation of section \ref{sec:param}, we observe that the nonlinear operator $(\mF(\cdot),\mB \cdot)$ maps $\mX=\mU(p,l)$ to $\mY=\mR(p,l)$ for appropriately defined spaces $\mU$ and $\mR$ for $(\mA,\mB)=(\mF'(v_0),\mB)$.

Let us pause on the definition of the boundary condition $g$ for $v_0+w$. We cannot expect $\mB v_0 =g$ with $g$ known so that $\mB w =0$ on $\partial X$. The reason is that $v_0$ is typically constructed by guessing $\gamma_0$ and solving the linear elliptic problems for $\{u_{j,0}\}$ with imposed Dirichlet conditions. It is for such a construction of $v_0$ that we were able to show that $A=(\mA,\mB)$ above was injective for $\mA=\mF'(v_0)$. The boundary condition $g_0:=\mB v_0$ thus partially depends on solving the above problem and is not in general given by the measurements $g$ (unless $v_0$ is the solution $v$). We thus recast the (modified) nonlinear hybrid inverse problem as
\begin{equation}
\label{eq:nonlinearmod}
\begin{array}{rcl}
 \mF (v_0+w) &=& \mH:=\mT(v_0) \tilde\mH \quad \mbox{ in } X \\
 \mB w &=& g-g_0 \quad \mbox{ on } \partial X.
\end{array}
\end{equation}
The objective of the following section is to provide an iterative algorithm to reconstruct $w$ provided that $v_0$ is sufficiently close to the solution we wish to obtain in the sense that $g-g_0$ and $\mH-\mH_0$ are sufficiently small with $\mH_0:=\mF(v_0)$. In section \ref{sec:injectivity}, we obtain an injectivity result stating that if $\mF(v_0+w)=\mF(v_0+\tilde w)$ and $\mB w=\mB \tilde w$, then $w=\tilde w$.

\begin{remark}
\label{rem:powerdensity} 

In the power density setting, we recast \eqref{eq:ellj} and \eqref{eq:pdj} as
\begin{displaymath}
   \mF_{2j-1}(v_0+w) = 0,\qquad \tilde \mF_{2j}(v_0+w)=H_j,\qquad 1\leq j\leq J,
\end{displaymath}
respectively. Let then $L_\gamma$, $F_j$, and $M_j$ constructed from $v_0=(\gamma,\{u_j\})$ as in the derivation of \eqref{eq:pdlinsyst2}. Finally, let us define
\begin{displaymath}
  \mF_{2j}(v_0+w) = L_\gamma \tilde \mF_{2j-1}(v_0+w) - M_j \mF_{2j-1}(v_0+w) = |F_j|^{-2} H_j =: K_j,\quad 1\leq j\leq J.
\end{displaymath}
The system $\mF_{2j-1}(v_0+w)=0$ and $\mF_{2j}(v_0+w)=K_j$ for $1\leq j\leq J$ is recast as 
\begin{displaymath}
  \mF(v_0+w) = \mH,
\end{displaymath}
in the notation of \eqref{eq:nonlinearmod} and implicitly defines the linear operator $\mT(v_0)$. We denote by $\mA=\mF'(v_0)$ the linearization of $\mF$ at $v_0$, which agrees with the differential operator defined in \eqref{eq:pdlinsyst2}.

Whereas the operator $\tilde \mB$ maps $v_0$ to the traces of $\{u_j\}$ on $\partial X$, the extended operator $\mB$ maps $v_0$ to the traces of $v_0$ and $\nu\cdot \nabla v_0$ on $\partial X$.

\end{remark}

\subsection{Iterative fixed point and reconstruction procedure}
\label{sec:fixedpoint}

Let us define
\begin{equation}
\label{eq:taylormF}
\mF(v_0+w) = \mF(v_0) + \mF'(v_0) w + \mG(w;v_0),
\end{equation}
where $\mG(w;v_0)$ is quadratic in the first variable in the sense that 
\begin{equation}
\label{eq:quadG}
  \|(\mG(w;v_0),0)\|_\mY \leq C \|w\|_\mX^2.
\end{equation}
The latter estimate comes from the fact that $\mF(v)$ is polynomial in $v$ and the partial derivatives $D$ and that $\mX$ is an algebra. We may thus recast the nonlinear system of equations for $w$ as
\begin{equation}
\label{eq:nonlinearw}
\begin{array}{rcl}
 \mA w &=& \mH - \mH_0 - \mG(w;v_0) \quad \mbox{ in } X \\
 \mB w &=& g-g_0 \quad \mbox{ on } \partial X.
\end{array}
\end{equation}
Since the linear operator $A=(\mA,\mB)$ is invertible, we may recast the above equation as
\begin{equation}
\label{eq:integralw}
 w = \mI(w) :=A^{-1}(\mH - \mH_0, g-g_0) - A^{-1}(\mG(w;v_0),0) .
\end{equation}
From the polynomial structure of $\mF$and the boundedness of $A^{-1}$ from $\mY$ to $\mX$, we deduce in addition to \eqref{eq:quadG} that
\begin{equation}
\label{eq:constants}
\begin{array}{rcl}
  \|A^{-1}(\mG(w;v_0),0)\|_{\mX} &\leq & C_1 \|w\|_\mX^2\\
  \|\mI(w)-\mI(\tilde w)\|_\mX &\leq& C_2 (\|w\|_{\mX} + \|\tilde w\|_{\mX}) \|w-\tilde w\|_{\mX}.
\end{array}
\end{equation}
As a consequence, for $\|w\|_\mX\leq R$ and $\|A^{-1}(\mH-\mH_0,g-g_0)\|_\mX\leq \eta$ sufficiently small so that $\eta+C_1R^2\leq R$ and $2C_2 R<1$, we find that $\mI$ is a contraction from the ball $B=\{w,\, \|w\|_{\mX}\leq R\}$ onto itself.

This shows that for $(\mH-\mH_0,g-g_0)$ sufficiently small, then the solution to \eqref{eq:nonlinearw} is unique and is obtained by the converging algorithm $w^{k+1}=\mI(w^k)$ initialized with $w^0=0$.

\subsection{Injectivity results}
\label{sec:injectivity}

The fixed point algorithm of the preceding section provides us with an explicit local reconstruction procedure for the nonlinear hybrid inverse problem. A similar methodology allows us to obtain local uniqueness results that are more general but not constructive. Using the same notation as in the preceding section, let us assume that 
\begin{equation}
\label{eq:uniqF}
\mF(v_0+w) = \mF(v_0+\tilde w) = \mH\quad\mbox{ in } X,\qquad \mB w=\mB \tilde w\mbox{ on }\partial X,
\end{equation}
for a fixed $v_0\in\mX$. In other words, the measurements associated with $v_0+w$ and $v_0+\tilde w$ are identical. Injectivity of the nonlinear problem $\mF$ locally in the vicinity of $v_0$ then means proving that $w=\tilde w$.

Since $\mF$ is a polynomial in $w$, then there is a polynomial function in $w$ and $\tilde w$ such that 
\begin{equation}
\label{eq:diffF}
\mF(v_0+\tilde w) - \mF(v_0+w) = \mJ(w,\tilde w) \cdot (\tilde w -w).
\end{equation}

When $\tilde w-w$ is sufficiently small, then $\mJ(w,\tilde w)$ satisfies the same properties as $\mF'(v_0)$, the Fr\'echet derivative of $\mF$. For instance, the ellipticity and unique continuation properties of $\mF'(v_0)$ still hold for $\mJ(w,\tilde w)$ for $w$ and $\tilde w$ small in $\mX$. As a consequence, we obtain that \eqref{eq:uniqF} implies that $w=\tilde w$ and hence that $\mF$ is injective. More generally, $\mJ(w,\tilde w)$ may still admit a left-inverse for $w$ and $\tilde w$ not necessarily close to each-other, in which case we also deduce that $w=\tilde w$ for $w$ and $\tilde w$ not necessarily small.

Let us assume more generally that we have two measurements
\begin{displaymath}
   \mF(v_0+w) = \mH,\quad \mB w=g-g_0 ;\qquad \mF(v_0+\tilde w) = \tilde\mH,\quad \mB\tilde w=\tilde g-g_0.
\end{displaymath}
Assume that $\mF'(v_0)$ is elliptic and injective. Then, $\mJ(w,\tilde w)$ is elliptic and hence admits a left-inverse (since it is injective), at least for $w$ and $\tilde w$ sufficiently close to $0$. As a consequence,  we obtain as above the stability estimate
\begin{equation}
\label{eq:stabnonlinear}
  \|w-\tilde w\|_{\mX} \leq C \|(\mH-\tilde\mH,g-\tilde g)\|_{\mY}.
\end{equation}

Note that the nonlinear problem \eqref{eq:uniqF} is invertible generically. Indeed, for $v_0$ analytic, $\mA^t\mA$ augmented with Dirichlet boundary conditions $\mD$ is invertible. As a consequence, the above result shows that \eqref{eq:uniqF} may be inverted for $(v_0+w)$ in an open set including $v_0$. Since analytic coefficients $v_0$  in $X_0$ restricted to $X$ are dense in the set of sufficiently smooth coefficients on $X$, we obtain that the inverse problem may be inverted {\em generically}; see \cite{SU-JFA-09}. 
When $v_0$ is not analytic, then we need to find a unique continuation principle based on Carleman estimates to obtain an estimate of the form \eqref{eq:stabnonlinear} for the fully nonlinear hybrid inverse problem.

\section*{Acknowledgment} 
I would like to thank the organizers of the Irvine conference for their invitation to a meeting that was a befitting tribute to the exceptional contributions to analysis and inverse problems of Gunther Uhlmann. I am glad to acknowledge many insightful discussions with Shari Moskow on the content of section 4, which borrows some ideas of our joint work in \cite{BM-LINUMOT-13}. I am also indebted to  Thomas Widlak for noticing and proposing solutions to inconsistencies in the presentation of the Lopatinskii conditions in section 2.6.2.
This paper was partially funded by NSF grant DMS-1108608 and AFOSR Grant NSSEFF- FA9550-10-1-0194.

\appendix
\section{Unique continuation for redundant systems}
\label{sec:uniqCalderon}

\subsection{Redundant system for a scalar function}

Let $x$ be a point in $\Rm^{n+1}$ and $N$ be a unit vector equal to $(0,\ldots,0,1)$ in an appropriate system of coordinates. We want to address the uniqueness of the Cauchy problem. Assume that $L_q$ for $1\leq q\leq Q$ are operators of order $m$ and that
\begin{equation}
\label{eq:redundsyst:app}
 L_q u =0 \mbox{ in } V\cap \{x_{n+1}>0\},\qquad \partial^j_{n+1} u=0 \mbox{ on } V\cap \{x_{n+1}=0\},
\end{equation}
for $1\leq q\leq Q$ and $1\leq j\leq m-1$, 
where $V$ is a neighborhood of $0$. Here, $u$ is a {\em scalar} function solution of a {\em redundant} system of equations. We assume that $N$ is non characteristic for all $L_q$. This implies, as in section \ref{sec:lur}, that all derivatives of $u$ vanish on $\{x_{n+1}=0\}$ and that $u$ can be extended by $0$ on $V\cap \{x_{n+1}<0\}$. The uniqueness problem may therefore be recast as: if $u$ satisfies
\begin{equation}
\label{eq:UCPsyst:app}
 L_q u =0 \mbox{ in } V \quad \mbox{ for } 1\leq q\leq Q\qquad \mbox{ and } \qquad u=0 \mbox{ in } \{x_{n+1}<0\},
\end{equation}
then $u\equiv0$ in a full neighborhood of $0$. 

When $Q=1$, it is known that $L=L_1$ needs to satisfy some restrictive assumptions in order for the result to hold; see \cite{H-SP-63,L-AM-57,N-AMS-73,Z-Birk-83}. The main advantage of the redundancy in the above system is that such assumptions need to be valid only locally (in the Fourier domain) for each operator, and globally collectively in a sense defined below. Our derivations closely follow the presentation in \cite{N-AMS-73}, to which we refer to additional details.

Changing notation, we define $t=x_{n+1}$ and still call $x=(x_1,\ldots,x_n)$. We then define $p_q=p_q(x,t,\xi,\tau)$ as the principal symbol of $L_q$ for this choice of coordinates. Sufficient conditions for the uniqueness to the above Cauchy problem involve the properties of the roots $\tau$ of the above polynomials $\tau\mapsto p_q(\cdot,\tau)$ as a function of $\xi$. In the setting of redundant measurements, these conditions may hold for different values of $q$ depending on the value of $\xi$. This justifies the following definition. 

We assume the existence of a a finite covering $\{\Omega_\nu\}$  of the unit sphere $\Sm^{n-1}$ (corresponding to $|\xi|=1$) such that the following holds. For each $\nu$, there exists $q=q(\nu)$ and $\eps>0$ such that for each $(x,t)$ close to $0$ and each $\xi\in\Omega_\nu$, we have
\begin{equation*}
\label{eq:condCalderon:app}
\begin{array}{rl}
(i) & \tau\to p_q(x,t,\xi,\tau) \mbox{  has at most simple real roots and at most double complex roots} \\
(ii) & \mbox{distinct roots $\tau_1$ and $\tau_2$ satisfy } |\tau_1-\tau_2|\geq\eps>0\\
(iii) & \mbox{non-real roots $\tau$ satisfy } |\Im \tau|\geq\eps.
\end{array}
\end{equation*}
Then we have the following result:
\begin{theorem}
\label{thm:calderonsystem:app}[Calder\'on's result for redundant systems.]
Assume that $N$ is non characteristic for the operators $L_q$ at the origin and that for a finite covering $\{\Omega_\nu\}$ of the unit sphere, (i)-(ii)-(iii) above are satisfied.   Then \eqref{eq:UCPsyst:app} implies that $u=0$ in a full neighborhood of $0$. 
\end{theorem}
\begin{proof}
 The proof very closely follows that of Theorem 5 in \cite{N-AMS-73}. 
 Note that the proof there requires all coefficients to be of class $C^\infty$. It has been shown in \cite{H-MS-59} (see also \cite[Notes of Chapter 2]{Z-Birk-83}) that such results also hold under less restrictive assumptions on smoothness. We do not track such results and assume that the coefficients are sufficiently smooth so that the technology developed in  \cite{N-AMS-73} applies.
 
 The differences between the proof of our theorem and that in \cite{N-AMS-73} are mostly notational. To present these differences in a somewhat self-contained manner, we need to introduce  most of the notation in \cite{N-AMS-73}, to which we refer the reader for details that are not reproduced here.

 Let $u(x,t)$ be of class $C^m$ with support contained in $|x|\leq r$ and $0\leq t\leq T$. Define $w=e^{k(T-t)^2}$. Then there is a constant $C$ independent of $u$ such that for $r$, $T$, and $k^{-1}$ sufficiently small, we have the Carleman-type inequality
 \begin{equation}
\label{eq:carleman:app}
 \dsum_{|\alpha|<m} |||D^\alpha u|||^2 \leq C (k^{-1}+T^2) \dsum_{q=1}^Q ||| L_qu |||^2, 
\end{equation}
where we have defined
\begin{equation}\label{eq:deftn:app}
  ||| u |||^2 = \dint_0^T w(t) \|u\|^2 dt,\qquad w(t) = e^{k(T-t)^2},
\end{equation}
 with $\|\cdot\|$ the $L^2$ norm in the variable $x$.  It is shown in \cite[p.31]{N-AMS-73} that the theorem follows from the above inequality, which we now prove.
 
 Since $N$ is non characteristic for $p_q$, the leading term of $p_q$ in $\tau^m$ does not vanish and we may divide $p_q$ by such a term to consider leading symbols (still called $p_q$) of the form
 \begin{displaymath}
 p_q(x,t,\xi,\tau) = \tau^m + \dsum_{j=1}^m Q_{qj}(x,t,\xi) \tau^{m-j},
\end{displaymath}
with $Q_{qj}$ homogeneous polynomials in $\xi$ of degree $j$.

Let $\Lambda=(1-\Delta_{x})^{\frac12}$ be the operator with symbol $(1+|\xi|^2)^{\frac12}$. For $u$ with (small) compact support, we define 
\begin{displaymath}
  u_j = \Lambda^{m-j} D_t^{j-1} u,\quad 1\leq j\leq m.
\end{displaymath}
Let then $U=(u_1,\ldots,u_m)^t$. We may recast $L_q u=f_q$ as 
\begin{displaymath}
\begin{array}{l}
D_t u_j-\Lambda u_{j+1} =0 ,\quad j<m \\
D_t u_m + \dsum_{j=1}^m Q_{qj}(x,t,D_x) \Lambda^{1-j} u_{m-j+1} + R_q U  = f_q,
\end{array}
\end{displaymath}
for $R_q$ a pseudo-differential of order $0$ in $x$ while $Q_{qj}(x,t,D_x) \Lambda^{1-j}$ is a pseudo-differential of order $1$ in $x$. The leading symbol of the above system for each $q$ is therefore a $m\times m$ matrix $\tau I+h_q$ with $h_q$ homogeneous of degree $1$ in $\xi$. We denote by $H_q(t,x,D_x)$ the pseudo-differential operator with symbol $h_q(t,x,\xi)$. Note that ${\rm Det}(\tau I+h_q)=p_q$ so that the roots of the system $\tau I+h_q$ are those of $p_q$.

From the derivation in \cite[p.33 \& p.35]{N-AMS-73}, \eqref{eq:carleman:app} follows (for $k^{-1}$ and $T^2$ sufficiently small) from
\begin{equation}
\label{eq:carleman2:app}
|||U|||^2 \leq C (k^{-1}+T^2) \dsum_{q=1}^Q |||(D_t+H_q)U|||^2 + C\dint_0^T w(t) \|U\|^2_{-1} dt.
\end{equation}
It thus remains to analyze the terms $\|(D_t+H_q)U\|$. The main idea of Calder\'on in \cite{C-AJM-58,C-PSFD-62} is to diagonalize $h_q$ and analyze the corresponding diagonalized first-order system. For $\xi$ unit vectors close to some $\xi_0$, we find a non-singular smooth $m\times m$ matrix $r_q(x,t,\xi)$ such that for all $(x,t)$ close to $0$, 
\begin{displaymath}
  r_q h_q r_q^{-1} = J_q,
\end{displaymath}
with $J_q$ in Jordan form. From the assumption (i), we obtain that for $q=q(\nu)$, then $J_q$ is either a $1\times 1$ matrix $\lambda(x,t,\xi)$ or a $2\times 2$ matrix with $\lambda(x,t,\xi)$ on the diagonal, $0$ on the (2,1) entry, and $1$ on the (1,2) entry. Once $r_q$ is defined for $|\xi|=1$, we extend it smoothly to all $\xi$ as a homogeneous function of degree $0$. This may be done locally on cones in $|\xi|\geq1$ but not necessarily globally. Following \cite[p.34]{N-AMS-73}, the global obstructions may be overcome by replacing our finite covering $\{\Omega_\nu\}$ by a finer (sub)covering, which we still call $\{\Omega_\nu\}$ such that for each $\xi\in\Omega_\nu$, there is a non-singular smooth matrix $r_{q\nu}$ such that
\begin{displaymath}
   r_{q\nu} h_q r_{q\nu}^{-1} \mbox{ is in Jordan canonical form.} 
\end{displaymath}

We then define $\phi_\nu(\xi)$ a $C^\infty$ partition of unity subordinate to $\Omega_\nu$ such that  $\sum_{\nu} \phi^2_\nu(\xi)=1$ for $|\xi|=1$ and such that each $\phi_\nu$ is extended to $|\xi|\geq1$ smoothly as a homogeneous function of degree $0$ and such that $\sum_{\nu} \phi^2_\nu(\xi)=1$ for $|\xi|>1$. Finally, we define $\Phi_\nu$ as the corresponding pseudo-differential operator of order $0$ and $U_\nu:=\Phi_\nu U$.

Since $[\Phi_\nu,H_q]$ is an operator of order $0$, we verify as in \cite[p.35]{N-AMS-73} that \eqref{eq:carleman2:app} is a consequence of
\begin{equation}
\label{eq:carleman3:app}
||| U_\nu|||^2 \leq C (k^{-1}+T^2) ||| (D_t+H_{q(\nu)})U_\nu |||^2 + C\dint_0^T w(t) \|U\|^2_{-1} dt,
\end{equation}
which is to hold for each $\nu$. Note that the summation over $q$ in \eqref{eq:carleman2:app} is a clear consequence of us obtaining the above inequality for $q=q(\nu)$. For the rest of the derivation, $q=q(\nu)$.

Now, $H_q U_\nu$ is unchanged if $r_{q\nu}$ and $j_{q\nu}$ are changed outside of the support of $\phi_\nu$. We thus construct symbols $h_{\nu}$ that agree with $h_q$ on $\Omega_\nu$ but that are globally reducible to a Jordan canonical form.  Each $h_{\nu}$ is then extended to all $\xi$ by a smooth function still called $h_{\nu}$ that is homogeneous of degree $1$ for $|\xi|\geq1$. We thus have globally defined symbols such that $r_{\nu} h_{\nu} r_{\nu}^{-1}= j_{\nu}$. We denote by $H_{\nu}$, $J_{\nu}$, $R_{\nu}$, and $S_{\nu}$ the pseudo-differential operators of order $1$, $1$, $0$, $0$, respectively, corresponding to the symbols $h_{\nu}$, $j_{\nu}$, $r_{\nu}$, and $r^{-1}_{\nu}$. 

We finally define $V_\nu=R_\nu U_\nu$. The rest of the proof is then as in \cite{N-AMS-73}. Following \cite[p.35]{N-AMS-73}, we deduce that \eqref{eq:carleman3:app} is a consequence of 
\begin{equation}
\label{eq:carleman4:app}
||| V_\nu ||| ^2 \leq C(k^{-1}+T^2) |||(D_t+J_\nu)V_\nu|||^2.
\end{equation}
Now $J_\nu$ obtained from $p_{q(\nu)}$ for $\xi\in \Omega_\nu$ has diagonal entries of the form $A(t)+iB(t)$ with $A$ and $B$ pseudo-differential operators of order $1$ and $B$ elliptic or vanishing when $J_\nu$ is a $1\times1$ matrix. Thus \eqref{eq:carleman4:app} is a consequence of \cite[Lemma 2]{N-AMS-73} and this concludes the proof of the theorem.
\end{proof}

\subsection{Redundant system for vector-valued functions}

The proof of the above theorem provides a local control encapsulated in \eqref{eq:carleman:app}. Above, $u$ is a scalar function. Such results may be extended to systems when $u$ is vector-valued provided that the leading term of the system is diagonal. Indeed, \eqref{eq:carleman:app} may then be used component by component to provide a uniqueness result for the system. We now consider such an extension. 

Consider operators $L_{qr}$ for $1\leq q\leq Q$ and $1\leq r\leq R$ and the system
\begin{equation}
\label{eq:redundsyst2:app}
L_{qr} u_r =0 \,\,\mbox{ in } V\cap \{ x_{n+1}>0\},\qquad \partial^j_{n+1}u_r =0  \,\,\mbox{ on } V\cap \{x_{n+1}=0\},
\end{equation}
with summation over repeated indexes and holding for all $1\leq q\leq Q$ in the first equation and for all $1\leq j\leq m-1$ and all $1\leq r\leq R$ in the second equation. Here $V$ is still a neighborhood of $0$. We still call $t=x_{n+1}$ and $x=(x_1,\ldots,x_n)$ in an appropriate system with $N$ as before. We denote by $p_{qr}=p_{qr}(x,t,\xi,\tau)$ the principal symbol of $L_{qr}$ for this choice of coordinates. As above, this uniqueness problem may be recast as: Assume that
\begin{equation}
\label{eq:UCPredsyst:app}
 L_{qr} u_r =0 \mbox{ in } V \quad \mbox{ for } 1\leq q\leq Q\qquad \mbox{ and } \qquad u_r=0 \mbox{ in } \{x_{n+1}<0\},
\end{equation}
for $1\leq r\leq R$. Then $u_r\equiv0$ in a full neighborhood of $0$. 

We sort the equations for $1\leq q\leq Q$ such that the following holds. We assume that $\{1,\ldots,Q\}$ is decomposed into $R$ connected components $I_r$ of the form $l(r-1)+1\leq q\leq l(r)$ for a strictly increasing function $l(r)$ from $\{1,\ldots,R\}$ to $\Nm^*$ with $l(0):=-1$ and $l(R)=Q$. Then we assume that $p_{qr}$ is of order $m$ for $q\in I_r$ and of order at most $m-1$ otherwise. When $Q=R$, then $p_{qr}$ forms a determined systems with diagonal entries of order $m$ and extra-diagonal entries of order at most $m-1$. More generally, the matrix $p_{qr}$ is composed of operators of order $m$ that are block-diagonal with blocks of size $I_r\times 1$ and of other operators of order at most $m-1$. 

Then, as we did before, we assume the existence of a finite covering $\{\Omega_\nu\}$ of the unit sphere $\Sm^{n-1}$ such that the following holds. For each $\nu$ and $r$, there exists a $q=q(\nu,r)\in I_r$ such that for each $(x,t)$ close to $0$ and each $\xi\in\Omega_\nu$, we have
\begin{equation*}
\label{eq:condCalderonred:app}
\begin{array}{rl}
(i) & p_{qr}(x,t,\xi,\tau) \mbox{  has at most simple real roots and at most double complex roots \hspace{0cm}} \\
(ii) & \mbox{distinct roots $\tau_1$ and $\tau_2$ satisfy } |\tau_1-\tau_2|\geq\eps>0\\
(iii) & \mbox{non-real roots $\tau$ satisfy } |\Im \tau|\geq\eps.
\end{array}\hspace{-0cm}
\end{equation*}

Then we have the following result:
\begin{theorem}
\label{thm:calderonsystemred:app}[Calder\'ons result for Redundant systems.]
Assume that $N$ is non characteristic for the operators $L_{qr}$ for $q\in I_r$ at the origin and that for a finite covering $\{\Omega_\nu\}$ of the unit sphere, (i)-(ii)-(iii) above are satisfied.   Then \eqref{eq:UCPsyst:app} implies that $u=0$ in a full neighborhood of $0$. 
\end{theorem}
\begin{proof}
The proof is similar to that of the preceding theorem. It is based on the following Carleman estimate
\begin{equation}
\label{eq:carlemanred:app}
 \dsum_{r=1}^R \dsum_{|\alpha|<m} |||D^\alpha u_r|||^2 \leq C (k^{-1}+T^2) \dsum_{q=1}^Q \dsum_{r=1}^R||| L_{qr}u_r |||^2, 
\end{equation}
whose proof is the same as that of \eqref{eq:carleman:app}. We leave the details to the reader.
\end{proof}
\subsection{Redundant systems in upper-triangular form}

In some applications, the leading term in the system may not be diagonal (or block-diagonal) but rather upper-triangular. In such a setting, unique continuation properties may still be valid provided that (a sufficiently large number of) the diagonal operators are elliptic. However, the corresponding complex roots may not longer be double. We have the following result, which as above, is stated for possibly redundant systems:
\begin{theorem}
\label{thm:calderonelliptic:app}
Assume that $N$ is non characteristic for the operators $L_q$ at the origin and that there exists a finite covering $\{\Omega_\nu\}$ of the unit sphere such that for each $\nu$, there exists $q=q(\nu)$ such that 
\begin{displaymath}
 p_q(x,t,\xi,\tau) \mbox{ has at most simple complex roots $\tau$ satisfying } |\Im \tau|\geq\eps
\end{displaymath}
for some $\eps>0$. Then there exists a constant $C$ independent of $u$ such that for $k^{-1}$ and $T$ sufficiently small, we have
\begin{equation}
\label{eq:carlemanelliptic:app}
  \dsum_{|\alpha|\leq m} |||D^\alpha u|||^2 \leq C (k^{-1}+T^2) \dsum_{q=1}^Q ||| L_qu |||^2.
\end{equation}
\end{theorem}
\begin{proof}
  The main difference with respect to the preceding theorem is that we have a full control of the $m$th-order derivatives of $u$. We follow the proof of Theorem \ref{thm:calderonsystem:app} until the very last step. The additional control of $D^\alpha u$ for $|\alpha|=m$ comes from the control of $\Lambda U_\nu$ and $D_t U_\nu$, or equivalently, that of $\Lambda V_\nu$ and $D_t V_\nu$. Now, $h_{q(\nu)}$ is diagonalized with simple eigenvalues $\lambda(x,t,\xi)=-a(x,t,\xi)-ib(x,t,\xi)$. By assumption $b(x,t,\xi)$ is uniformly bounded away from $0$. We may then apply \cite[(7.12)]{N-AMS-73} to obtain that 
  \begin{displaymath}
   |||\Lambda V_\nu|||^2 + |||D_t V_\nu|||^2 \leq C(1+kT^2) ||| (D_t+J_{\nu})V_\nu|||^2.
\end{displaymath}
The proof of \eqref{eq:carlemanelliptic:app} follows.
\end{proof}
As an application of the preceding result, let us consider the following corollary:
\begin{theorem}
\label{thm:2by2system:app}
Consider the redundant system of equations
\begin{equation}
\label{eq:2by2syst:app}
\left(\begin{matrix} L_1 & L_0 \\ L_3 & L_2 \end{matrix} \right) 
\left(\begin{matrix} u_1 \\ u_2\end{matrix} \right) = 0\mbox{ in } V\cap \{x_{n+1}>0\},\qquad \partial^j_{n+1} u_k=0 \mbox{ on } V\cap \{x_{n+1}=0\},
\end{equation}
for $1\leq j\leq m-1$ and $1\leq k\leq 2$, with the following assumptions. The operators $L_1$ and $L_0$ are (vector-valued) $Q_1\times1$ operators of order $m$,  where $L_1$ satisfies the hypotheses of Theorem \ref{thm:calderonsystem:app} with $Q=Q_1$. The operators $L_2$ and $L_3$ are $Q_2\times1$ operators of order $m$ and at most $m-1$, respectively. Moreover, $L_2$ satisfies the hypotheses of Theorem \ref{thm:calderonelliptic:app} with $Q=Q_2$.
Then $(u_1,u_2)=0$ in a full neighborhood of $0$.

The same result extends systems of the form
\begin{equation}
\label{eq:RbyRsyst:app}
\left(L_{ij} \right)_{1\leq i,j\leq R} u = 0\mbox{ in } V\cap \{x_{n+1}>0\},\qquad \partial^j_{n+1} u_k=0 \mbox{ on } V\cap \{x_{n+1}=0\},
\end{equation}
for $1\leq k\leq R$, where $L_{ij}$ is (a vector-valued operator) of order $m-1$ for $i>j$, $L_{ii}$  is (a vector-valued operator) of order $m$ that satisfies the hypotheses of Theorem \ref{thm:calderonsystem:app} with an appropriate value of $Q$ when all $L_{ik}$ are of order $m-1$ for $k\not=i$ and the hypotheses of  Theorem \ref{thm:calderonelliptic:app} with an appropriate value of $Q$ when at least one operator $L_{ik}$ for $k<i$ is of order $m$.
\end{theorem}
\begin{proof}
First of all, since $N$ is non characteristic for $L_1$ and $L_2$, we observe that all derivatives of $u_k$ vanish on $\{x_{n+1}\}$ so that both functions can be extended by $0$ for $\{x_{n+1}<0\}$. The above system of PDEs is thus satisfied on $V$ with $u_k=0$ for $x_{n+1}<0$ and we wish to show that $u_k=0$ for $x_{n+1}\geq0$.

The proof is based on a modification of the method showing Theorem \ref{thm:calderonsystem:app} from \eqref{eq:carleman:app}. Indeed, we claim that for $(u_1,u_2)$ with support contained in $|x|\leq r$ and $0\leq t\leq T$, then the theorem follows from
\begin{equation}
\label{eq:carleman2by2:app}
\dsum_{|\alpha|<m}\dsum_{k=1}^2 |||D^\alpha u_k|||^2 \leq C(k^{-1}+T^2) |||Lu|||^2,
\end{equation}
where we denote by $Lu$ the left-hand side in \eqref{eq:2by2syst:app}. The proof is a straightforward modification of \cite[p.31]{N-AMS-73}, which we reproduce here. Let $\zeta$ be a smooth function defined for $t\geq0$ equal to $1$ for $t\leq 2T/3$ and $0$ for $t\geq T$. Let $u_k=\zeta v_k$ for $v=(v_1,v_2)$ solution of $Lv=0$. Then from \eqref{eq:carleman2by2:app}, using only the estimate with $|\alpha|=0$,  we obtain that
\begin{displaymath}
 \dint_0^{\frac{2T}3} \|v\|^2 w dt \leq C(k^{-1}+T^2) \dint_{\frac{2T}3}^{T} \|L(\zeta v)\|^2 w dt \leq C' (k^{-1}+T^2)  \dint_{\frac{2T}3}^{T} w dt,
\end{displaymath}
with $C'$ independent of $k$. This shows that 
\begin{displaymath}
 e^{kT^2/4} \dint_{0}^{\frac T2} \|v\|^2 dt \leq C' (k^{-1}+T^2) T e^{k T^2/9},
\end{displaymath} 
and as $k\to\infty$ that $\int_0^{\frac T2} \|v\|^2 dt\to0$.

It remains to show \eqref{eq:carleman2by2:app}. Using Theorem \ref{thm:calderonsystem:app}, we obtain that the left-hand-side is bounded by
\begin{displaymath}
   C(k^{-1}+T^2)\big( |||L_1u_1|||^2 +|||L_2u_2|||^2 \big)\leq
   C(k^{-1}+T^2) \big( ||| Lu |||^2 + ||| L_0 u_2 |||^2 + ||| L_3 u_1|||^2 \big).
\end{displaymath}
Now from Theorem \ref{thm:calderonelliptic:app}, we obtain that
\begin{displaymath}
   ||| L_0 u_2 |||^2 \leq C (k^{-1}+T^2) |||L_2 u_2|||^2.
\end{displaymath}
Moreover, we obtain that 
\begin{displaymath}
   ||| L_3 u_1|||^2  \leq C \dsum_{|\alpha|<m} |||D^\alpha u_1|||^2.
\end{displaymath}
Therefore, for $(k^{-1}+T^2)$ sufficiently small, we deduce \eqref{eq:carleman2by2:app}.
The generalization to \eqref{eq:RbyRsyst:app} is mostly notational and is left to the reader.
\end{proof}


\begin{thebibliography}{10}

\bibitem{adams}
{\sc R.~A. Adams}, {\em {Sobolev Spaces}}, Academic Press, NY, San Francisco,
  London, 1975.

\bibitem{ADN-CPAM-59}
{\sc S.~Agmon, A.~Douglis, and L.~Nirenberg}, {\em {Estimates near the boundary
  for solutions of elliptic partial differential equations satisfying general
  boundary conditions. I}}, Comm. Pure Appl. Math., 12 (1959), pp.~623--727.

\bibitem{ADN-CPAM-64}
\leavevmode\vrule height 2pt depth -1.6pt width 23pt, {\em {Estimates near the
  boundary for solutions of elliptic partial differential equations satisfying
  general boundary conditions. II}}, Comm. Pure Appl. Math., 17 (1964),
  pp.~35--92.

\bibitem{A-Sp-08}
{\sc H.~Ammari}, {\em An Introduction to Mathematics of Emerging Biomedical
  Imaging}, vol.~62 of Mathematics and Applications, Springer, New York, 2008.

\bibitem{ABCTF-SIAP-08}
{\sc H.~Ammari, E.~Bonnetier, Y.~Capdeboscq, M.~Tanter, and M.~Fink}, {\em
  Electrical impedance tomography by elastic deformation}, SIAM J. Appl. Math.,
  68 (2008), pp.~1557--1573.

\bibitem{AS-IP-12}
{\sc S.~R. Arridge and O.~Scherzer}, {\em Imaging from coupled physics},
  Inverse Problems, 28 (2012), p.~080201.

\bibitem{B-IO-12}
{\sc G.~Bal}, {\em {Hybrid inverse problems and internal functionals}}, Inside
  Out II, MSRI Publications, G. Uhlmann Editor, Cambridge University Press,
  Cambridge, UK, 2012.

\bibitem{B-APDE-13}
\leavevmode\vrule height 2pt depth -1.6pt width 23pt, {\em {Cauchy problem for
  Ultrasound modulated EIT}}, To appear in Analysis \& PDE. {\tt
  arXiv:1201.0972v1},  (2013).

\bibitem{BBMT-13}
{\sc G.~Bal, E.~Bonnetier, F.~Monard, and F.~Triki}, {\em Inverse diffusion
  from knowledge of power densities}, Inverse Problems and Imaging, 7(2)
  (2013), pp.~353--375.

\bibitem{BC-JDE-13}
{\sc G.~Bal and M.~Courdurier}, {\em Boundary control of elliptic solutions to
  enforce local constraints}, J. Differential Equations, 255(6) (2013),
  pp.~1357--1381.

\bibitem{BGM-IP-13}
{\sc G.~Bal, C.~Guo, and F.~Monard}, {\em Inverse anisotropic conductivity from
  internal current densities}, Submitted.

\bibitem{BGM-IPI-13}
\leavevmode\vrule height 2pt depth -1.6pt width 23pt, {\em Linearized internal
  functionals for anisotropic conductivities}, Submitted.

\bibitem{BM-LINUMOT-13}
{\sc G.~Bal and S.~Moskow}, {\em Local inversions in ultrasound modulated
  optical tomography}, Submitted, {\tt arXiv:1303.5178},  (2013).

\bibitem{BR-IP-11}
{\sc G.~Bal and K.~Ren}, {\em {Multi-source quantitative PAT in diffusive
  regime}}, Inverse Problems, 27(7) (2011), p.~075003.

\bibitem{BRUZ-IP-11}
{\sc G.~Bal, K.~Ren, G.~Uhlmann, and T.~Zhou}, {\em Quantitative
  thermo-acoustics and related problems}, Inverse Problems, 27(5) (2011),
  p.~055007.

\bibitem{BS-PRL-10}
{\sc G.~Bal and J.~C. Schotland}, {\em {Inverse Scattering and Acousto-Optics
  Imaging}}, Phys. Rev. Letters, 104 (2010), p.~043902.

\bibitem{BU-CPAM-13}
{\sc G.~Bal and G.~Uhlmann}, {\em Reconstruction of coefficients in scalar
  second-order elliptic equations from knowledge of their solutions}, to appear
  in CPAM {arXiv:1111.5051}.

\bibitem{BU-IP-10}
\leavevmode\vrule height 2pt depth -1.6pt width 23pt, {\em Inverse diffusion
  theory for photoacoustics}, Inverse Problems, 26(8) (2010), p.~085010.

\bibitem{BU-AML-12}
\leavevmode\vrule height 2pt depth -1.6pt width 23pt, {\em Reconstructions for
  some coupled-physics inverse problems}, Applied Math. Letters, 25(7) (2012),
  pp.~1030--1033.

\bibitem{C-AJM-58}
{\sc A.~Calder{\'o}n}, {\em {Uniqueness in the Cauchy problem for Partial
  Differential Equations}}, Amer. J. Math., 80(1) (1958), pp.~16--36.

\bibitem{C-PSFD-62}
\leavevmode\vrule height 2pt depth -1.6pt width 23pt, {\em {Existence and
  uniqueness theorems for systems of partial differential equations}}, Proc.
  Sympos. Fluid Dynamics and Appl. Math. (University of Maryland, 1961), Gordon
  and Breach, New York,  (1962), pp.~147--195.

\bibitem{CFGK-SJIS-09}
{\sc Y.~Capdeboscq, J.~Fehrenbach, F.~{de Gournay}, and O.~Kavian}, {\em
  Imaging by modification: numerical reconstruction of local conductivities
  from corresponding power density measurements}, SIAM J. Imaging Sciences, 2
  (2009), pp.~1003--1030.

\bibitem{C-JMAA-91}
{\sc C.~Cosner}, {\em On the definition of ellipticity for systems of partial
  differential equations}, J. Math. Anal. Appl., 158 (1991), pp.~80--93.

\bibitem{CAB-IP-07}
{\sc B.~T. Cox, S.~R. Arridge, and P.~C. Beard}, {\em Photoacoustic tomography
  with a limited-apterture planar sensor and a reverberant cavity}, Inverse
  Problems, 23 (2007), pp.~S95--S112.

\bibitem{CLB-SPIE-09}
{\sc B.~T. Cox, J.~G. Laufer, and P.~C. Beard}, {\em The challenges for
  quantitative photoacoustic imaging}, Proc. of SPIE, 7177 (2009), p.~717713.

\bibitem{DN-CPAM-55}
{\sc A.~Douglis and L.~Nirenberg}, {\em Interior estimates for elliptic systems
  of partial differential equations}, Comm. Pure Appl. Math., 8 (1955),
  pp.~503--538.

\bibitem{GS-SIAP-09}
{\sc B.~Gebauer and O.~Scherzer}, {\em Impedance-acoustic tomography}, SIAM J.
  Applied Math., 69(2) (2009), pp.~565--576.

\bibitem{H-MS-59}
{\sc L.~V. H{\"o}rmander}, {\em {On the uniqueness of the Cauchy problem II}},
  Math. Scand., 7 (1958), pp.~177--190.

\bibitem{H-SP-63}
\leavevmode\vrule height 2pt depth -1.6pt width 23pt, {\em Linear Partial
  Differential Operators}, Springer-Verlag, Berlin, 1963.

\bibitem{H-I-SP-83}
\leavevmode\vrule height 2pt depth -1.6pt width 23pt, {\em {The Analysis of
  Linear Partial Differential Operators I: Distribution Theory and Fourier
  Analysis}}, Springer Verlag, 1983.

\bibitem{H-III-SP-94}
\leavevmode\vrule height 2pt depth -1.6pt width 23pt, {\em {The Analysis of
  Linear Partial Differential Operators III: Pseudo-Differential Operators}},
  Springer Verlag, 1994.

\bibitem{KK-EJAM-08}
{\sc P.~Kuchment and L.~Kunyansky}, {\em Mathematics of thermoacoustic
  tomography}, Euro. J. Appl. Math., 19 (2008), pp.~191--224.

\bibitem{KK-AET-11}
\leavevmode\vrule height 2pt depth -1.6pt width 23pt, {\em {2D and 3D
  reconstructions in acousto-electric tomography}}, Inverse Problems, 27
  (2011), p.~055013.

\bibitem{KS-IP-12}
{\sc P.~Kuchment and D.~Steinhauer}, {\em Stabilizing inverse problems by
  internal data}, Inverse Problems, 28 (2012), p.~084007.

\bibitem{LR-JPCS-11}
{\sc J.~Le~Rousseau}, {\em {On Carleman estimates with two large parameters}},
  J. Phys. Conf. Ser., 290 (2011), p.~012010.

\bibitem{L-AM-57}
{\sc H.~Lewy}, {\em {An Example of a Smooth Linear Partial Differential
  Equation Without Solution}}, Ann. of Math., 66 (1957), pp.~155--158.

\bibitem{MY-IP-04}
{\sc J.~R. McLaughlin and J.~Yoon}, {\em Unique identifiability of elastic
  parameters from time-dependent interior displacement measurement}, Inverse
  Problems, 20 (2004), p.~25–45.

\bibitem{MZM-IP-10}
{\sc J.~R. McLaughlin, N.~Zhang, and A.~Manduca}, {\em Calculating tissue shear
  modulus and pressure by 2{D} log-elastographic methods}, Inverse Problems, 26
  (2010), pp.~085007, 25.

\bibitem{MB-aniso-13}
{\sc F.~Monard and G.~Bal}, {\em Inverse anisotropic conductivity from power
  density measurements in dimensions $n\geq3$}, submitted.

\bibitem{MB-IP-12}
\leavevmode\vrule height 2pt depth -1.6pt width 23pt, {\em Inverse anisotropic
  diffusion from power density measurements in two dimensions}, Inverse
  Problems, 28 (2012), p.~084001.

\bibitem{MB-IPI-12}
\leavevmode\vrule height 2pt depth -1.6pt width 23pt, {\em Inverse diffusion
  problem with redundant internal information}, Inverse Problems and Imaging,
  6(2) (2012), pp.~289--313.

\bibitem{NTT-Rev-11}
{\sc A.~Nachman, A.~Tamasan, and A.~Timonov}, {\em Current density impedance
  imaging}, Tomography and Inverse Transport Theory. Contemporary Mathematics
  (G. Bal, D. Finch, P. Kuchment, P. Stefanov, G. Uhlmann, Editors), 559
  (2011).

\bibitem{N-CPAM-57}
{\sc L.~Nirenberg}, {\em {Uniqueness in Cauchy problems for differential
  equations with constant leading coefficients}}, Comm. Pure Appl. Math., 10
  (1957), pp.~89--105.

\bibitem{N-AMS-73}
\leavevmode\vrule height 2pt depth -1.6pt width 23pt, {\em {Lectures on Linear
  Partial Differential Equations}}, Amer. Math. Soc., Providencem R.I., 1973.

\bibitem{PS-IP-07}
{\sc S.~Patch and O.~Scherzer}, {\em Photo- and thermo- acoustic imaging},
  Inverse Problems, 23 (2007), pp.~S1--10.

\bibitem{S-SP-2011}
{\sc O.~Scherzer}, {\em Handbook of Mathematical Methods in Imaging}, Springer
  Verlag, New York, 2011.

\bibitem{SW-SR-11}
{\sc J.~K. Seo and E.~J. Woo}, {\em {Magnetic Resonance Electrical Impedance
  Tomography (MREIT)}}, SIAM Review, 53 (2011), pp.~40--68.

\bibitem{S-JSM-73}
{\sc V.~A. Solonnikov}, {\em Overdetermined elliptic boundary-value problems},
  J. Sov. Math., 1 (1973), pp.~477--512.

\bibitem{SU-JFA-09}
{\sc P.~Stefanov and G.~Uhlmann}, {\em Linearizing non-linear inverse problems
  and an application to inverse backscattering}, J. Funct. Anal., 256 (2009),
  pp.~2842--2866.

\bibitem{SU-IO-12}
\leavevmode\vrule height 2pt depth -1.6pt width 23pt, {\em Multi-wave methods
  by ultrasounds}, Inside out, Cambridge University Press (G. Uhlmann, Ed.),
  (2012).

\bibitem{T-AV-95}
{\sc N.~N. Tarkharov}, {\em {The Cauchy problem for solutions of elliptic
  equations}}, Akademie Verlag, Berlin, 1995.

\bibitem{WW-W-07}
{\sc L.~V. Wang and H.~Wu}, {\em {Biomedical Optics: Principles and Imaging}},
  Wiley, 2007.

\bibitem{Z-Birk-83}
{\sc C.~Zuily}, {\em Uniqueness and Non-Uniqueness in the Cauchy Problem},
  Birkh\"auser, Boston, 1983.

\end{thebibliography}
%
%
%

\end{document}